\documentclass[reqno]{amsart}
\usepackage{amsmath,amsthm,amscd,amssymb,amsfonts, amsbsy}
\usepackage{latexsym, color, enumerate}
\usepackage{pxfonts}
\usepackage{marginnote}
\usepackage{todonotes}

\usepackage{tikz}
\usetikzlibrary{snakes}

\usepackage{hyperref}

\newtheorem{theorem}{Theorem}[section]

\newtheorem{lemma}[theorem]{Lemma}
\newtheorem{proposition}[theorem]{Proposition}

\theoremstyle{definition}
\newtheorem{remark}[theorem]{Remark}

\theoremstyle{definition}

\theoremstyle{definition}
\newtheorem{assumption}[theorem]{Assumption}

\theoremstyle{definition}

\newcommand{\dv}{\operatorname{div}}

\newcommand{\supp}{\operatorname{supp}}
\newcommand{\dist}{\operatorname{dist}}
\newcommand{\diam}{\operatorname{diam}}

\numberwithin{equation}{section}

\newcommand{\bR}{\mathbb{R}}

\newcommand\cD{\mathcal{D}}

\newcommand\cH{\mathcal{H}}
\newcommand\cL{\mathcal{L}}
\newcommand\cN{\mathcal{N}}

\newcommand\cQ{\mathcal{Q}}

\newcommand\cM{\mathcal{M}}

\newcommand\vN{\vec{N}}

\providecommand{\norm}[1]{\lVert#1\rVert}

\renewcommand{\vec}[1]{\boldsymbol{#1}}

\def\XXint#1#2#3{{\setbox0=\hbox{$#1{#2#3}{\int}$}
		\vcenter{\hbox{$#2#3$}}\kern-.5\wd0}}

\newcommand{\p}{\partial}
\newcommand{\epsi}{\varepsilon}

\begin{document}



	\title[mixed boundary value problem]{The Dirichlet-conormal problem for the heat equation with inhomogeneous boundary conditions}

	\author[Hongjie Dong]{Hongjie Dong}	
	
	\address{
Division of Applied Mathematics, Brown University, 182 George Street, Providence, RI 02912, USA}
	
	\email{hongjie\_dong@brown.edu}
\thanks{H. Dong was partially supported by a Simons fellowship grant no.$\,$007638, the NSF under agreement DMS-2055244, and the Charles Simonyi Endowment at the Institute of Advanced Study}

	\author[Zongyuan Li]{Zongyuan Li}
	
	\address{Department of Mathematics, Rutgers University, Hill Center - Busch Campus, 110 Frelinghuysen Road, Piscataway, NJ 08854, USA}
	
	\email{zongyuan.li@rutgers.edu}
\thanks{Z. Li was partially supported by an AMS-Simons travel grant.}
	
\begin{abstract}
We consider the mixed Dirichlet-conormal problem for the heat equation on cylindrical domains with a bounded and Lipschitz base $\Omega\subset \bR^d$ and a time-dependent separation $\Lambda$. Under certain mild regularity assumptions on $\Lambda$, we show that for any $q>1$ sufficiently close to 1, the mixed problem in $L_q$ is solvable. In other words, for any given Dirichlet data in the parabolic Riesz potential space $\mathcal{L}_q^1$ and the Neumann data in $L_q$, there is a unique solution and the non-tangential maximal function of its gradient is in $L_q$ on the lateral boundary of the domain. When $q=1$, a similar result is shown when the data is in the Hardy space.
Under the additional condition that the boundary of the domain $\Omega$ is Reifenberg-flat and the separation is locally sufficiently close to a Lipschitz function of $m$ variables, where $m=0,\ldots,d-2$, with respect to the Hausdorff distance, we also prove the unique solvability result for any $q\in(1,(m+2)/(m+1))$. In particular, when $m=0$, i.e., $\Lambda$ is Reifenberg-flat of co-dimension $2$, we derive the $L_q$ solvability in the optimal range $q\in (1,2)$.
For the Laplace equation, such results were established in \cite{OB21,OB13,BC20} and \cite{DL20}.
\end{abstract}

\maketitle

\section{Introduction}

We consider the heat equation on a cylinder $\cQ=(-\infty,\infty)\times\Omega$ with the base $\Omega\subset \bR^d$ being bounded and Lipschitz. The lateral boundary of $\cQ$ (denoted as $\p_l\cQ$) is decomposed into two non-intersecting (open) components $\cD$ and $\cN$ separated by $\Lambda$ satisfying
$$
{\overline{\cD}}\cup\cN=\p_l\cQ,\quad \cD\cap\cN=\emptyset,\quad \Lambda=\overline{\cD}\cap\overline{\cN}.
$$
See Figure \ref{pic-domains}.
\begin{figure}
\begin{tikzpicture}
	\draw [->] (-2,-3.3) -- (-2, 0) node (taxis) [right] {$t$};
	\draw [->] (-2,-3.3) -- (2,-3.3) node (xaxis) [above] {$x$};
	\draw (0,-0.2) ellipse (1.5 and 0.2);
	\draw (-1.5,-0.2) -- (-1.5,-3);
	\draw[snake=zigzag, red] (0,-0.4) -- (0,-3.2);
	\draw (-1.5,-3) arc (180:360:1.5 and 0.2);
	\draw [dashed] (-1.5,-3) arc (180:360:1.5 and -0.2);
	\draw (1.5,-0.2) -- (1.5,-3);  
	\fill [gray,opacity=0.3] (-1.5,-0.2) -- (-1.5,-3) arc (180:360:1.5 and 0.2) -- (1.5,-0.2) arc (0:180:1.5 and -0.2);
	\node[left] at (0,-0.6) {$\Lambda$};
	\node at (1,-1.8) {$\cD$};
	\node at (-0.8,-1.8) {$\cN$};
	\node at (0.8,-0.2) {$\cQ$};
\end{tikzpicture}
\label{pic-domains}
\caption{Cylinder with a time-varying $\Lambda$.}
\end{figure}
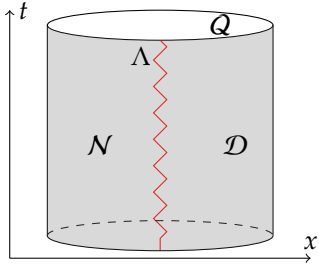
On such cylinder, for some $q\in[1,\infty)$ and boundary data $g_\cN \in L_q$ (replaced with the Hardy space $H^1$ when $q=1$) and $g_\cD\in \cL^1_q$, we consider the initial boundary value problem
\begin{equation}		\label{eqn-main}
	\begin{cases}
		u_t-\Delta u = 0  & \text{in }\, \cQ^T,\\
		\frac{\p u}{\p \vec{n}} = g_\cN  & \text{on }\, \cN^{T},\\
		u = g_\cD & \text{on }\, \cD^{T},\\
		u = 0 & \text{on }\,\{t=0\},\\
		\vec{N}(Du) \in L_q(\p_l\cQ^T).
	\end{cases}
\end{equation}
Here, for $T\in(0,\infty)$, we set
\begin{equation}\label{eqn-211105-0426}
	\cQ^T:=\{(t,x)\in \cQ : t\in(0,T)\}.
\end{equation}
Similarly, we define $\p_l\cQ^T$, $\cD^T$, $\cN^T$, and $\Lambda^T$.
The space $\cL^1_q$, roughly, include functions which together with spatial tangential derivatives and the half time derivative are all in $L_q$. For the precise definition, see \eqref{eqn-210827-0635}.

We call $u$ a solution to \eqref{eqn-main}, if $u\in \cH^{1}_{2, loc}(\cQ^T)$ is a weak solution in the usual sense, and the non-tangential maximal function of $Du$ is controlled. For such solution, the non-tangential limits of $u$ and $Du$ exist a.e. at the boundary. Hence the boundary conditions can be understood as the non-tangential limit.

The mixed Dirichlet-Neumann problems for elliptic and parabolic equations received a lot of attention, from both application and theoretical aspects. For instance, when modeling melting procedure of a piece of floating ice, one need to deal with zero Dirichlet boundary condition below the water level and zero Neumann boundary condition above the water level. Also, such problems were commonly studied in the combustion theory and in modelling exocytosis, which have a form of active transport mechanism. For applications, see \cite{CDL21} and the references therein.

The interests from the theoretical side mainly come from the low regularity feature of such problems. Let us start from elliptic (Laplace) equations. Unlike the purely Dirichlet or Neumann problem, solutions to mixed problem can be non-smooth near the interfacial boundary $\Lambda$, even when the domain and boundary values are all smooth. In particular, by examining the classical example on $\bR^2_+$:
\begin{equation*}
	u(x,y)=\operatorname{Im}(x+iy)^{1/2},
\end{equation*}
one can see that $Du\vert_{\p\bR^2_+}\in L_{2-\epsi,\text{loc}}$, but not in $L_{2,\text{loc}}$. This causes the failure of the solvability method in \cite{MR890159} for purely Dirichlet or Neumann problems, which starts from the $L_2$ solvability in \cite{JK81}. Great efforts have been made for the solvability of mixed problems. Initiated in \cite{B}, Brown studied the $L_2$ solvability with the extra geometric assumption that $\cD$ and $\cN$ meet at an cute angle. Such assumption excludes the smooth domains, but makes the aforementioned $L_2$ method applicable. See also \cite{MR2309180}. On general Lipschitz domains, Brown et. al. obtained the $L_q$ solvability for $q$ sufficiently close to $1$ in \cite{TOB}, under very weak assumptions on $\Lambda$ -- merely Ahlfors regular of Hausdorff dimension close to $d-2$ and the so-called corkscrew condition on $\cD$. Later, in \cite{BC20} an explicit solvability range $q<d/(d-1)$ was obtained, assuming $\p\Omega\in C^{1,1}$ and $\Lambda$ to be Lipschitz.

For the Laplace equation, people are also interested in the minimum assumptions on $\p\Omega$ and $\Lambda$ such that the optimal $L_{2-\epsi}$ solvability holds. In \cite{DL20}, we improved the result in \cite{BC20} and obtained the $L_q$ solvability when $q\in (1,(m+2)/(m+1))$, under the conditions that $\p\Omega$ is Lipschitz and Reifenberg flat, and $\Lambda$ is a small perturbation of a Lipschitz graph in $m$ variables, where $m=0,1,\ldots,d-2$. In particular, when $m=0$, $\Lambda$ is Reifenberg flat of codimension $2$, and the optimal solvability range $q\in (1,2)$ was achieved.

Now let us turn to heat equations. In the case of purely Dirichlet or Neumann problems on a Lipschitz cylinder $(0,T)\times\Omega$, Brown proved the $L_2$ and the optimal $L_q$ solvability for $q\in [1,2+\epsi)$ in \cite{B89} and \cite{B90}. These are the parabolic analog of the elliptic results in \cite{JK81} and \cite{MR890159}. See also the estimates \cite{MR2280778} in Besov spaces on Lipschitz cylinders. The problems on time-varying domains are more involved, as it is well known that merely the parabolic Lipschitz assumption on $\p_l\cQ$ is not enough to guarantee the solvability. By assuming certain extra smoothness of $\p_l\cQ$ in the $t$-variable, some (optimal) $L_2$ and $L_q$ solvability results were proved in \cite{MR1323804, HL96,HL05}. See also the discussions of equations with variable coefficients in \cite{MR4127944, MR3809457} and the references therein.

However, there are very few results regarding parabolic equations with mixed boundary conditions. See \cite{CDL21} and the references therein. Here, we consider such problems on domains with some time dependency -- Lipschitz cylinders with time-varying $\Lambda$. More precisely, when the base $\Omega$ is Lipschitz,  under mild conditions on $\Lambda$, we obtain the $L_1$ and $L_q$ solvability for $q$ sufficiently close to $1$. Furthermore, based on our earlier results under homogeneous boundary conditions in \cite{CDL21}, when $\p\Omega$ is also Reifenberg flat and $\Lambda$ is close to a Lipschitz graph in $m$ variables in the sense of Hausdorff distance, the solvability range can be extended to $q\in (1,(m+2)/(m+1))$, which achieves the aforementioned optimal range when $m=0$. 

In the current paper, besides the aforementioned low regularity issue, the time-dependency also brings tremendous difficulties. For instance, certain integral identities and the harmonic conjugates used in \cite{OB13, TOB} are not available in the parabolic setting. Also, we do not have a good control of $u_t$ close to $\Lambda$. Such difficulties are overcome by a duality argument combined with the De Giorgi-Nash-Moser estimate, and a dedicated decomposition in \cite{CDL21}. After the current project, we also plan to study more complicated time varying domains.

\section{Notations and main results}
\subsection{Notations and function spaces}\label{sec-211025-0917}
For 
any cylinder $\widetilde{\cQ}=(S,T)\times\omega$, we denote its parabolic and lateral boundaries as
\begin{equation*} \p_p\widetilde{\cQ}:=(\{S\}\times\omega)\cup\big([S,T)\times\p\omega\big),\quad \p_l \widetilde{\cQ}:=(S,T)\times\p\omega.
\end{equation*}
For two points $X=(t,x)$ and $Y=(s,y)$ in $\bR^{1+d}$, we denote their parabolic distance as
\begin{equation*}
	\operatorname{dist}(X,Y)=|X-Y|:=|x-y| + |t-s|^{1/2}=\big(\sum_{i=1}^d |x^i-y^i|^2\big)^{1/2} + |t-s|^{1/2}.
\end{equation*}
Furthermore, let
$$
d(X):=\dist(X,\partial_l\cQ)=\dist(x,\Omega),\quad
\delta(X):=\dist(X,\Lambda).
$$
It is convenient to work with parabolic cubes and surface cubes with center $X=(t,x^1,\ldots,x^d)$
\begin{equation}
	\label{eq10.27}
	\begin{split}
		Q_r(X) := (t-r^2,t+r^2)\times \prod_{i=1}^d(x^i-r,x^i+r),\\
		\cQ_r(X) := Q_r(X)\cap\cQ,\quad \Delta_r(X) := Q_r(X)\cap\p_l\cQ.
	\end{split}
\end{equation}
As usual, the center will be omitted when it is the origin. The parabolic non-tangential cone (restricted to $\cQ^T$) with aperture $\alpha>0$ and the corresponding non-tangential maximal function for $f\in L_{1,\text{loc}}(\cQ^T)$ 
are defined to be
\begin{equation*}
	\begin{split}
		&\Gamma(X):=\{Y\in\cQ^T: |Y-X|\leq (1+\alpha)\operatorname{dist}(Y,\p_l\cQ)\},\\
		&\vec{N}(f)(X):=\sup_{Y\in\Gamma(X)}|f(Y)|
	\end{split}
\end{equation*}
for $X=(t,x)\in\p_l\cQ^T$.
We also define their truncated versions as
\begin{equation}\label{eqn-210918-0600}
	\Gamma_r(X):=\Gamma(X_0)\cap \cQ_r(X),\quad \vec{N}_{r}(f):=\sup_{Y\in \Gamma_r(X)}|f(Y)|,\quad\text{and}\,\,\vec{N}^{r}(f):=\sup_{Y\in \Gamma(Y)\setminus\Gamma_r(X)}|f(Y)|.
\end{equation}
In the paper we will choose $\alpha$ large enough. As usual, since the particular value of $\alpha$ only influences the constants in the estimates, we omit the dependence on it as the norms are comparable.


Next, we introduce some function spaces. First, we define spaces for weak solutions. As in \cite{CDL21}, by $u\in \mathbb{H}^{-1}_{p}(\cQ)$ we mean that there exist $g=(g_1,\ldots, g_d)\in L_{p}(\cQ)^d$ and $f\in L_{p}(\cQ)$ such that
$$
u=D_i g_i+f \quad \text{in }\, \cQ, \quad g_in_i=0 \quad \text{on }\, \cN
$$
holds in the distribution sense,
where $n=(n_1,\ldots,n_d)$ is the outward unit normal to $\partial \Omega$,
and the norm
$$
\|u\|_{\mathbb{H}^{-1}_{p}(\cQ)}=\inf\big\{ \|g\|_{L_{p}(\cQ)}+\|f\|_{L_{p}(\cQ)} : u=D_ig_i+f\, \text{ in }\, \cQ, \, g_in_i=0 \, \text{ on }\, \cN\big\}
$$
is finite.
We set
$$
\cH_{p}^1(\cQ)=\big\{u: u\in W^{0,1}_{p}(\cQ), \, u_t\in \mathbb{H}^{-1}_{p}(\cQ)\big\}
$$
equipped with a norm
$$
\|u\|_{\cH^1_{p}(\cQ)}
=\|u\|_{W^{0,1}_{p}(\cQ)}+\|u_t\|_{\mathbb{H}^{-1}_{p}(\cQ)}.
$$
Here $W^{k,l}_p(\cQ)$ is the usual Sobolev space with $\p_t^\alpha\p_x^\beta u\in L_p$ when $\alpha\leq k$ and $|\beta|\leq l$. Now we define the spaces for boundary data. On $\p_l\cQ = \cD\cup\cN$, we call a function $a$  an atom, if
\begin{align*}
	\supp(a)\subset\Delta_r\,\,\text{for some}\,\, r<R_0/2,\quad
	\norm{a}_{L_\infty}\leq \frac{1}{|\Delta_r|},\quad
	\fint_{\Delta_r}a = 0\quad\text{when}\,\,\Delta_r\subset\cN.
\end{align*}
Then, $H^1(\p_l\cQ)$ is defined to be the $l^1$-span of the atoms. More precisely, $f\in H^1(\p_l\cQ)$ if and only if there exists atoms $\{a_j\}$ and real numbers $\{\lambda_j\}$ with $\sum|\lambda_j|<\infty$ such that $f=\lambda_j a_j$. The norm is defined as
\begin{equation*}
	\|f\|_{H^1(\p_l\cQ)}:=\inf\sum_j|\lambda_j|,
\end{equation*}
where the infimum is taken with respect to all such decompositions.
We define $H^1(\cN)$ to be the restriction of $H^1(\p_l\cQ)$, with the norm being defined as the infimum among all extensions.

We also need to define spaces for the Dirichlet data, for which we also need to include the tangential derivative and the ``half time derivative''. On the hyperplane $\bR\times\bR^{d-1}$, denote a typical point $X'=(t,x')$. We define the parabolic Riesz potential of the first order
\begin{equation*}
	I^{par}_1(g)(X') = \int_{-\infty}^t \int_{\bR^{d-1}} \frac{2}{(4\pi (t-s))^{d/2}}e^{-|x'-y'|^2/(4(t-s))}g(s,y')\,dy'ds.
\end{equation*}
Then we define
\begin{equation}\label{eqn-210827-0635} \mathcal{L}_p^1(\bR\times\bR^{d-1}):=\big\{f:f=I^{par}_1(g),\, g\in L_p(\bR\times\bR^{d-1})\big\},\quad\|f\|_{\mathcal{L}_p^1} = \|g\|_{L_p}.
\end{equation}
From this, we can define the space $\mathcal{L}^1_p(S)$ when $S=\bR\times\{x^1=\varphi(x')\}$ is a Lipschitz graph, by projection. Furthermore, the space $\mathcal{L}^1_p(\p_l\cQ) = \mathcal{L}^1_p(\bR\times\p\Omega)$ can be defined via a partition of unity, and hence $\mathcal{L}^1_p(\cD)$ by restriction. When $p=1$, $\mathcal{L}^1_1$ is defined by replacing $\|g\|_{L_p}$ with $\|g\|_{H^1}$ in \eqref{eqn-210827-0635}.

The following properties can be found in \cite[pp.6]{B90}:
\begin{equation*}
	\|f\|_{\mathcal{L}^1_p}\approx \|\p_t I_1^{par}(f)\|_{L_p} + \sum_i\|\p_{x^i}f\|_{L_p}
\end{equation*}
and when $p=2$,
\begin{equation*}	\|f\|_{\mathcal{L}^1_2}^2
	\approx\int_{\bR^{d-1}}\int_{\bR}|\hat{f}(\tau,x')|^2|\tau|\,d\tau dx'
	+\sum_i\|\p_{x^i}f\|_{L_2}^2,
\end{equation*}
where $\hat{f}$ is the Fourier transform in $t$. This means, roughly speaking, $\mathcal{L}^1_p$ includes functions with a half time derivative and the whole (tangential) spatial derivatives in $L_p$.

In the above, we define functions on infinite-long cylinders ($t\in\bR$). In order to solve the initial Dirichlet problem, we also need to take the compatibility condition into consideration. For a smooth function $f$ on $[S,T]\times \partial\Omega$ satisfying $f(S,\cdot)=0$ on $\partial\Omega$, we also define
\begin{equation*} \|f\|_{\cL^1_p((S,T)\times\p\Omega)}:=\|\tilde{f}\|_{\cL^1_p(\bR\times\p\Omega)},
\end{equation*}
where
\begin{equation}\label{eqn-211024-0621}
	\tilde{f}(t,x):=
	\begin{cases}
		0\,\,\text{when}\,\,t\in(-\infty,S)\cup(2T-S,\infty),\\
		f\,\,\text{when}\,\,t\in(S,T),\\
		f(2T-t,x)\,\,\text{when}\,\,t\in(T,2T-S).
	\end{cases}
\end{equation}
Furthermore, the space $\cL^1_p((S,T)\times\p\Omega)$ is defined as the completion of all smooth functions which equal to zero when $t=S$. As usual, $\cL^1_p(\cD)$ is defined via restriction. For details about $\cL^1_p$ spaces, we refer the reader to \cite{B90}.

%
\subsection{Assumptions and main results}
Let $R_0>0$ be a fixed small constant. Throughout the paper, we impose the following Lipschitz assumption on $\partial \Omega$.
\begin{assumption}[$(M,R_0)$-Lipschitz]\label{ass-small-Lip}
There exists some constant $R_0>0$, such that, for any $x_0\in \p\Omega$, there exists a coordinate system $x=(x^1,x')$ and a Lipschitz function $\psi_0:\bR^{d-1}\rightarrow \bR$ such that
\begin{equation*}
\Omega_{R_0}(x_0):= \Omega \cap B_{R_0}(x_0) = \{x \in B_{R_0}(0) : x^1 > \psi_0(x')\}\quad\mbox{and}\quad  |D\psi_0(x')|<M \quad \mbox{a.e. }
\end{equation*}
\end{assumption}

Clearly, certain assumptions on $\Lambda$ are also needed for the regularity of solutions.  Assumption \ref{ass-0301-2356} requires that $\Lambda, \cD$, and $\cN$ all vary continuously in time.
\begin{assumption}\label{ass-0301-2356}
	For any $\epsi>0$ and $L>0$, there exist a time partition
	$$
	-L=t_0<t_1<\cdots<t_{n}=L
	$$
	and decompositions $\p\Omega=D^{t_k}\cup N^{t_k}$, $k\in \{1,\ldots, n\}$, such that
	\begin{equation*}
		\cD(t)\subset D^{t_k},\quad H^d(\cD(t), D^{t_k})<\epsi,\quad \forall t\in[t_{k-1},t_k),
	\end{equation*}
	where $\cD(t)=\{x\in \partial \Omega: (t,x)\in \cD\}$ and $H^d$ is the usual $d$-dimensional Hausdorff distance.
\end{assumption}
In particular, Assumption \ref{ass-0301-2356} excludes the possibility shown in Figure \ref{pic-impossible-domain}.
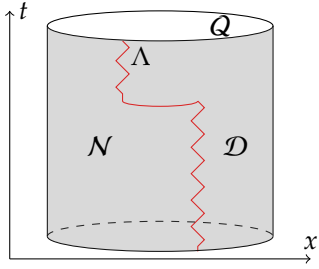
\begin{figure}[!hb]
		\begin{tikzpicture}
			\draw [->] (-2,-3.3) -- (-2, 0) node (taxis) [right] {$t$};
			\draw [->] (-2,-3.3) -- (2,-3.3) node (xaxis) [above] {$x$};
			\draw (0,-0.2) ellipse (1.5 and 0.2);
			\draw (-1.5,-0.2) -- (-1.5,-3);
			\draw[snake=zigzag, red] (-0.5,-0.4) -- (-0.5,-1.2);
			\draw[red] (-0.5,-1.2) arc (180:360:0.5 and 0.07);
			\draw[snake=zigzag, red] (0.5,-1.2) -- (0.5,-3.2);
			\draw (-1.5,-3) arc (180:360:1.5 and 0.2);
			\draw [dashed] (-1.5,-3) arc (180:360:1.5 and -0.2);
			\draw (1.5,-0.2) -- (1.5,-3);  
			\fill [gray,opacity=0.3] (-1.5,-0.2) -- (-1.5,-3) arc (180:360:1.5 and 0.2) -- (1.5,-0.2) arc (0:180:1.5 and -0.2);
			\node[left] at (0,-0.6) {$\Lambda$};
			\node at (1,-1.8) {$\cD$};
			\node at (-0.8,-1.8) {$\cN$};
			\node at (0.8,-0.2) {$\cQ$};
		\end{tikzpicture}
	\label{pic-impossible-domain}
	\caption{Impossible shape of $\Lambda$.}
\end{figure}

Roughly speaking, Assumption \ref{ass-210730-0540} below gives information about the Hausdorff dimension of $\Lambda$. It is satisfied, for example, when $\Lambda$ is Lipschitz or Assumption \ref{ass-210609-0500-3} holds for a small constant $\theta$. Here, $A$ and $R_0$ are fixed constants, and $\epsi_0$ is a parameter to be chosen small later.
\begin{assumption}[$\epsi_0$]\label{ass-210730-0540}
	For any $s>-1+\epsi_0$, $r<R_0$, and any point $X\in\Lambda$,
	\begin{align}
		\int_{Q_r(X)\cap\p_l\cQ}\delta^s\,d\sigma\approx r^{d+1+s},\label{eqn-210622-0316-1}\\
		\int_{Q_r(X)\cap\cQ}\delta^{s-1}\,dZ\approx r^{d+1+s}.\label{eqn-210622-0316-2}
	\end{align}
Here $\delta(X)=\dist(X,\Lambda)$ and $Q_r(X)$ is a parabolic cube defined in \eqref{eq10.27}.
\end{assumption}

	\begin{remark}
	Actually, such integrability conditions are also what essentially used in the elliptic cases \cite{TOB, OB13}. See \cite[Lemma~2.4]{TOB}.
\end{remark}

Assumption \ref{ass-210809-0731} requires that near $\Lambda$, the Dirichlet boundary $\cD$ cannot be too small locally. For example, cusps are not allowed.
\begin{assumption}[Corkscrew]\label{ass-210809-0731}
	There exists some $M>0$, such that for all $X_0\in\Lambda$ and $R\in(0,R_0]$, there exists a surface cube $\Delta_{R/M}(X)\subset \cD$ with $|X-X_0|<R$. For the definition of surface cubes, see \eqref{eq10.27}.
\end{assumption}

In the next assumption, it is assumed that $\p\Omega$ is locally flat and $\Lambda$ is locally close to a Lipschitz graph.
\begin{assumption}[$\theta, m$]		\label{ass-210609-0500-3}
	Let $m\in \{0,1,\ldots, d-2\}$ and $M\in (0, \infty)$.
	\begin{enumerate}[$(a)$]
		\item
		For any $x_0\in \partial \Omega$ and $R\in (0, R_0]$, there is a (spatial) coordinate system depending on $x_0$ and $R$ such that in this coordinate system, we have
		\begin{equation}		\label{200429@eq1}
			\{y: y^1>x_0^1+\theta R\}\cap B_R(x_0)\subset \Omega_R(x_0)\subset \{y: y^1>x_0^1-\theta R\}\cap B_R(x_0).
		\end{equation}
		
		\item
		For any $X_0=(t_0,x_0)\in \Lambda$ and $R\in (0, R_0]$,   there exist
		a (spatial) coordinate system and a Lipschitz function $\phi$ of $m$ variables with Lipschitz constant $M$,  such that in the new coordinate system (called the coordinate system associated with $(X_0, R)$), we have  \eqref{200429@eq1},
		$$
		\big(\partial_{l} \cQ\cap Q_R(X_0) \cap \{(s,y): y^2>\phi(y^3,\ldots,y^{m+2})+\theta R\} \big)\subset \cD,
		$$
		$$
		\big(\partial_{l} \cQ \cap Q_R(X_0)\cap  \{(s,y): y^2<\phi(y^3,\ldots, y^{m+2})-\theta R\}\big)\subset \cN,
		$$
		and
		$$
		\phi(x_0^3,\ldots, x_0^{m+2})=x_0^2.
		$$
		Here, if $m=0$, then the function $\phi$ is  understood as the constant function $\phi\equiv x_0^2$.
	\end{enumerate}
\end{assumption}
It is not difficult to see that when $\theta$ is sufficiently small, Assumption \ref{ass-210609-0500-3} ($\theta, m$) implies Assumptions \ref{ass-0301-2356} and \ref{ass-210809-0731}. In Appendix \ref{app-210919-1049}, we show how Assumption \ref{ass-210609-0500-3} (b) with a small $\theta$ implies Assumption \ref{ass-210730-0540} with a small $\epsi_0$.

Now we are ready to state our main results.
\begin{theorem}\label{thm-210830-0419}
	Suppose that $\Lambda$ satisfies Assumptions \ref{ass-0301-2356} and \ref{ass-210809-0731}. There exists $\epsi_0 =\epsi_0(d,M)>0$ small enough, such that if Assumption \ref{ass-210730-0540} is satisfied, we have the following.
	\begin{enumerate}
		\item ($L_1$-existence) For any $g_{\cN}\in H^1(\cN^T)$ and $g_{\cD}\in \mathcal{L}^1_1(\cD^T)$, there exists a solution to \eqref{eqn-main} with $q=1$, satisfying
		\begin{equation}\label{eqn-210910-0549}
			\norm{\vec{N}(Du)}_{L_1(\p_l\cQ^T)}\leq C(\norm{g_{\cN}}_{H^1(\cN^T)} + \norm{g_\cD}_{\mathcal{L}^1_1(\cD^T)}),
		\end{equation}
where $C=C(d,R_0,M,\diam(\Omega), T)$.
		\item ($L_1$-uniqueness) Suppose that $u$ is a solution to \eqref{eqn-main} with $q\geq 1$, $g_\cN=0$, and $g_\cD=0$, then we must have $u\equiv 0$.
		\item ($L_{1+\epsi}$ solvability). There exist some $\epsi_0=\epsi_0(d,M)>0$ small enough, such that for any $\varepsilon\in (0,\varepsilon_0)$, if $g_{\cN}\in L_{1+\epsi}(\cN^T)$ and $g_{\cD}\in \mathcal{L}^1_{1+\epsi}(\cD^T)$, then the unique solution $u$ obtained in (a) satisfies $\vec{N}(Du)\in L_{1+\epsi}(\p_l\cQ^T)$, with
		\begin{equation*}
			\norm{\vec{N}(Du)}_{L_{1+\epsi}(\p_l\cQ^T)}\leq C(\norm{g_{\cN}}_{L_{1+\epsi}(\cN^T)} + \norm{g_\cD}_{\mathcal{L}^1_{1+\epsi}(\cD^T)}),
		\end{equation*}
where $C=C(d,R_0,M,\diam(\Omega), T,\varepsilon)$.
		\end{enumerate}
\end{theorem}
Furthermore, with the ``local flatness'' assumption, we have the optimal regularity.
\begin{theorem}\label{thm-211009-0913}
	For any $q\in(1,(m+2)/(m+1))$, there exists $\theta=\theta(d,m,q)>0$ sufficiently small, such that if Assumption \ref{ass-210609-0500-3} ($\theta,m$) is satisfied, then for any $g_\cN\in L_q(\cN^T)$ and $g_\cD\in \mathcal{L}^1_q(\cD^T)$, there exists a unique solution to \eqref{eqn-main}, satisfying
	\begin{equation}\label{eqn-210901-0628}
		\norm{\vec{N}(Du)}_{L_q(\p_l\cQ^T)}\leq C(\norm{g_{\cN}}_{L_q(\cN^T)} + \norm{g_\cD}_{\mathcal{L}^1_q(\cD^T)}),
	\end{equation}
where $C=C(d,R_0,M,\diam(\Omega), T, q)$.
\end{theorem}

The rest of the paper is organized as follows. First, we prove the existence of $\cH^1_2$-weak solutions and a global $L_2$ estimate in Section \ref{sec1.1}. In Section \ref{sec1.2} - \ref{sec-211031-1136}, we prove some basic weighted and unweighted local estimates for pure Dirichlet, pure Neumann, and mixed problems. With all these preparations, we prove the $L_1$ existence result in Theorem \ref{thm-210830-0419} (a) in Section \ref{sec-211018-0507}. The uniqueness in Theorem \ref{thm-210830-0419} (b) is proved in Section \ref{sec-210817-0523}. Finally, the higher regularity results in Theorem \ref{thm-211009-0913} and Theorem \ref{thm-210830-0419} (c) are proved in Section \ref{sec-211018-0509}.

\section{Preliminary}

\subsection{Existence of \texorpdfstring{$\cH^1_2$}{H12}-weak solutions}    \label{sec1.1}
To start with, we solve for $\cH^1_2$-weak solutions to \eqref{eqn-main}. In the following, let $\cH^{1/2}_2(\cD^T)$ be the restriction of the trace space of $\cH^1_2(\cQ^T)$ on $\cD^T\subset\p_l\cQ^T$. Similarly, $\cH^{-1/2}_2(\cN^T)$ is the trace space of the (spatial) normal derivatives for functions in $\cH^1_2(\cQ^T)$.
\begin{theorem}
                        \label{thm2.1}
	Suppose that $\Lambda$ satisfies Assumption \ref{ass-0301-2356}, then for any  $g_\cD\in  \cH^{1/2}_2(\cD^T)$ and $g_{\cN}\in \cH^{-1/2}_2(\cN^T)$, there exists a unique solution $u\in \cH^1_{2}(\cQ^T)$ to  \eqref{eqn-main} satisfying
	\begin{equation*}
		\|u\|_{\cH^1_2(\cQ^T)} \leq C(\|g_\cD\|_{\cH^{1/2}_2(\cD^T)} + \|g_\cN\|_{\cH^{-1/2}_2(\cN^T)}),
	\end{equation*}
where $C=C(d, R_0, M, T,\text{diam}(\Omega))$.
\end{theorem}
\begin{proof}
	By the solvability of the Dirichlet problem, without loss of generality, we may assume that $g_\cD=0$. Now for each $t\in (0,T)$, by the Lax-Milgram theorem, we can find $w(t,\cdot)\in W^1_{2,\cD}(\Omega)$ satisfying
	\begin{equation*}
		\begin{cases}
			\Delta w(t,\cdot) = 0  & \text{in }\, \Omega,\\
			\frac{\p w}{\p \vec{n}} = g_\cN  & \text{on }\, \cN(t),\\
			w = 0 & \text{on }\, \cD(t).
		\end{cases}
	\end{equation*}
	Let $f(t,\cdot)=\nabla w(t,\cdot)$ so that $f\in L_2(\cQ^T)$. Then \eqref{eqn-main} is equivalent to the following nonhomogeneous equation with ``homogeneous'' boundary condition
	\begin{equation*}		
		\begin{cases}
			u_t-\Delta u = -\dv f  & \text{in }\, \cQ^T,\\
			\frac{\p u}{\p \vec{n}} = f\cdot \vec{n}  & \text{on }\, \cN^T,\\
			u = 0 & \text{on }\, \cD^T,\\
			u = 0 & \text{on } \{t=0\},
		\end{cases}
	\end{equation*}
	the solvability of which follows from \cite[Proposition 3.3]{CDL21}.
\end{proof}

Next, using the embedding on $\p \cQ^T$ and the duality (cf. \cite[Lemma 4.13]{OB13}), we have the global $L_2$ estimate when $g_{\cD}=0$:
\begin{equation}
	\label{eq8.21}
	\|Du\|_{L_2(\cQ^T)}\le C\|g_{\cN}\|_{L_{\frac{2(d+1)}{d+2}}(\cN^T)},
\end{equation}
where $C=C(d,R_0,M,T)$.

\subsection{Reverse H\"older's inequality in \texorpdfstring{$\cQ$}{Q}} \label{sec1.2}
By using the Caccioppoli inequality and the parabolic Sobolev-Poincar\'e inequality (cf. \cite[Lemma 3.8]{CDL21}), we can prove a reverse H\"older's inequality ``in the bulk'', which is analogous to \cite[Lemma 3.19]{OB13}. Suppose that $0\in\p_l\cQ$. Consider a local problem
	\begin{equation}\label{eqn-210827-0729}
	\begin{cases}
		u_t-\Delta u = 0  & \text{in }\, \cQ_{2r},\\
		\frac{\p u}{\p \vec{n}} = g_\cN  & \text{on }\, \cN\cap\Delta_{2r},\\
		u = 0 & \text{on }\, \cD\cap\Delta_{2r},
	\end{cases}
\end{equation}
where $r\in (0,R_0/2)$.
\begin{lemma}\label{lem-210903-0918}
	Suppose that Assumptions \ref{ass-0301-2356} and \ref{ass-210809-0731} hold, and $u\in\cH^1_2$ satisfies \eqref{eqn-210827-0729}.
	\begin{enumerate}
		\item We can find some $p_0=p_0(d,M)>2$, such that
		\begin{equation}
			\label{eq6.54}
			\Big(\fint_{\cQ_r}|Du|^{p_0}\,dZ\Big)^{1/p_0}
			\le C\fint_{\cQ_{2r}}|Du|\,dZ
			+C\Big(\fint_{\Delta_{2r}}1_{\cN}
			|g_{\cN}|^{p_0(d+1)/(d+2)}\,d\sigma\Big)^{(d+2)/(p_0(d+1))},
		\end{equation}
where $C=C(d,M)$.
		\item For any $p_0< 2(m+2)/(m+1)$, if we further assume that Assumption \ref{ass-210609-0500-3} ($\theta,m$) is satisfied with a sufficiently small $\theta=\theta(d,M,p_0)$ and $g_\cN=0$, then \eqref{eq6.54} holds. In this case, the constant $C$ also depends on $m$ and $p_0$.
	\end{enumerate}
\end{lemma}
Here , (a) is the usual self-improving property, of which we omit the proof. For the assertion (b), we refer the reader to \cite[Lemma 6.1]{CDL21}.
It is worth mentioning that for Lipschitz cylinder $\cQ$, the corkscrew condition on $\cD$ relative to $\Lambda$ is enough when applying the Sobolev-Poincar\'e inequality in the mixed case with $u=0$ on $\cD$.

\subsection{Local estimates on \texorpdfstring{$\p_l\cQ$}{pomega} for pure Dirichlet/Neumann problems} \label{sec1.3}
\begin{lemma}\label{lem-210216-0333-1}
	Suppose that $\Delta_2\subset\cN$. There exists a constant $\epsi=\epsi(d,M)>0$, such that for any $p\in(1,2+\epsi)$ and any solution $u\in \cH^1_p$ to \eqref{eqn-210827-0729} with $g_{\cN}\in L_p$, we have
	\begin{equation}\label{eqn-211024-0805}
		\|Du\|_{L_p(\Delta_1)}\le \|\vN_{1/2}(D u)\|_{L_p(\Delta_1)} \le C \|g_{\cN}\|_{L_p(\Delta_2)} + C\|Du\|_{L_p(\cQ_2)},
	\end{equation}
where $C=C(d,R_0,M, p)$.
\end{lemma}
\begin{proof}
	We focus on estimating $\vec{N}_{1/2}(Du)$ as the first inequality in \eqref{eqn-211024-0805} is trivial. Without lost of the generality, we may assume $(u)_{\cQ_2} =0$. Now, we take a cut-off function $\eta=\eta(t,x)\in C^\infty_c(\cQ_2)$ with $\eta=1$ in $\cQ_{7/4}$. Then $u\eta$ satisfies an inhomogeneous heat equation
	\begin{equation*}
		\begin{cases}
			(\p_t-\Delta)(u\eta) = f\quad\text{in}\,\,\cQ_2,\\
			\p(u\eta)/\p \vec{n} = (g_\cN\eta + u\p \eta/\p\vec{n})1_{\Delta_2}\quad\text{on}\,\,\p_l\cQ_2,\\
			u\eta=0\quad\text{for}\,\,t=-4,
		\end{cases}
	\end{equation*}
where
	\begin{equation}\label{eqn-211024-0934}
		f:=(u\p_t\eta -u\Delta\eta - 2 Du\cdot D\eta)1_{\cQ_2} \in L_p.
	\end{equation}

To deal with the source terms,
on an enlarged smooth cylinder
\begin{equation*}
	\widetilde{Q}:=(-4,4)\times B_{2\sqrt{d}}\supset Q_2,
\end{equation*}
we solve the following problem
	\begin{equation}\label{eqn-211024-0935}
	\begin{cases}
		(\p_t-\Delta)w = f\quad\text{in}\,\,\widetilde{Q},\\
		 w = 0\quad\text{on}\,\,\p_p \widetilde{Q}.\\
	\end{cases}
	\end{equation}
	Since $f\in L_p(\widetilde{Q})$ and $\widetilde{Q}$ has a smooth base, such solution $w\in W^{1,2}_p(\widetilde{Q})$ exists, and satisfies
		\begin{equation}\label{eqn-211010-1130}
			\|w\|_{W^{1,2}_p(\widetilde{Q})}\leq \|f\|_{L_p(\widetilde{Q})}\leq C\|u\|_{L_p(\cQ_2)} + C\|Du\|_{L_p(\cQ_2)} \leq C\|Du\|_{L_p(\cQ_2)}.
		\end{equation}
	Here in the last inequality, we also used the parabolic Poincar\'e type inequality for functions with space-time zero mean, c.f., \cite[Lemma~3.8]{CDL21} or \cite[Lemma~3.1]{MR2304157}.
Furthermore, since $w$ is caloric in $Q_{7/4}$, by the interior regularity for caloric functions and \eqref{eqn-211010-1130},
	\begin{equation}\label{eqn-211010-1136}
		\|Dw\|_{L_\infty(Q_{3/2})} \leq C\|Dw\|_{L_p(Q_{7/4})} \leq C\|Du\|_{L_p(\cQ_2)}.
\end{equation}
Note that $u\eta-w$ is caloric in $\cQ_2$ and $u\eta-w = 0$ when $t=-4$.
Now by the triangle inequality, the fact
	\begin{equation*}
		\Gamma_{1/2}(X)\subset \cQ_{3/2}\quad\forall X\in\Delta_1,
	\end{equation*}
and the $L_p$ estimate for the initial-Neumann problem in \cite[Theorem~5.22,~5.25]{B90}, we have
	\begin{equation}\label{eqn-211010-1133}
		\begin{split}
		\|\vec{N}_{1/2}(Du)\|_{L_p(\Delta_1)}
		&\leq
		\|\vec{N}_{1/2}(D(u\eta -w))\|_{L_p(\Delta_1)} + \|\vec{N}_{1/2}(Dw)\|_{L_p(\Delta_1)}
		\\&\leq
		\|\vec{N}_{1/2}(D(u\eta -w))\|_{L_p(\Delta_1)} + C\|Dw\|_{L_\infty(Q_{3/2})}
		\\\quad&\leq
		C\big\|\frac{\p(u\eta -w)}{\p\vec{n}}\big\|_{L_p(\p_l\cQ_2)} + C\|Dw\|_{L_\infty(Q_{3/2})}.
	\end{split}
	\end{equation}
Noting that
	\begin{equation*}
		\p(u\eta)/\p \vec{n} = (g_\cN\eta + u\p \eta/\p\vec{n})1_{\Delta_2}\quad\text{on}\,\,\p_l\cQ_2,
	\end{equation*}
by the triangle inequality, we have
\begin{align}
		\|\frac{\p(u\eta -w)}{\p\vec{n}}\big\|_{L_p(\p_l\cQ_2)}
		&\leq
		C\|g_\cN\|_{L_p(\Delta_2)} + C\|u\|_{L_p(\Delta_2)} + \|Dw\|_{L_p(\p_l\cQ_2)}\nonumber
		\\&\leq
		C\|g_\cN\|_{L_p(\Delta_2)} + C\|u\|_{W^{0,1}_p(\cQ_2)} + \|w\|_{W^{0,2}_p(\cQ_2)}.\label{eqn-211024-0901}
\end{align}
Here, in \eqref{eqn-211024-0901}, we used the following inequality which can be derived from the trace theorem (in $x$): For any $h\in W^{0,1}_p(\cQ_2)$,
\begin{equation*}
	\|h\|_{L_p(\p_l\cQ_2)}^p = \int_{-4}^4 \|h(t,\cdot)\|_{L_p(\p\Omega_2)}^p\,dt \leq C\int_{-4}^4  \|h(t,\cdot)\|_{W^{0,1}_p(\Omega_2)}^p\,dt = C\|h\|_{W^{0,1}_p(\cQ_2)}^p.
\end{equation*}
Substituting \eqref{eqn-211024-0901} back to \eqref{eqn-211010-1133}, then using the Poincar\'e type inequality, \eqref{eqn-211010-1130}, and \eqref{eqn-211010-1136}, we have
\begin{equation*}
	\begin{split}
		\|\vec{N}_{1/2}(Du)\|_{L_p(\Delta_1)}
		&\leq
		C\|g_\cN\|_{L_p(\Delta_2)} + C\|u\|_{W^{0,1}_p(\cQ_2)} + \|w\|_{W^{0,2}_p(\cQ_2)} + \|Dw\|_{L_\infty(Q_{3/2})}
		\\&\leq C\|g_\cN\|_{L_p(\Delta_2)} + C\|Du\|_{L_p(\cQ_2)}.
	\end{split}
\end{equation*}
Hence, the lemma is proved.
\end{proof}

For the Dirichlet problem, some more work is needed to bound the half derivative in time.
\begin{lemma}\label{lem-210216-0333-2}
Suppose that $\Delta_2\subset\cD$. There exists a constant $\epsi=\epsi(d,M)>0$, such that for any $p\in(1,2+\epsi)$ and solution $u\in \cH^1_p$ to \eqref{eqn-210827-0729}, we have,
	\begin{equation*}
		\|Du\|_{L_p(\Delta_1)}\le \|\vN_{1/2}(D u)\|_{L_p(\Delta_1)} \le C\|Du\|_{L_p(\cQ_2)},
	\end{equation*}
where $C=C(d,R_0,M, p)$.
\end{lemma}
\begin{proof}
	Again, it suffices to estimate $\vec{N}_{1/2}(Du)$. Taking the same cutoff function $\eta$ with Lemma \ref{lem-210216-0333-1}, we have
	\begin{equation*}
	\begin{cases}
		(\p_t-\Delta)(u\eta) = f\quad\text{in}\,\,\cQ_2,\\
		u = 0\quad\text{on}\,\,\p_l\cQ_2,\\
		u\eta=0\quad\text{for}\,\,t=-4,
	\end{cases}
	\end{equation*}
where $f$ is given in \eqref{eqn-211024-0934}. Still, we solve \eqref{eqn-211024-0935} for $w\in W^{1,2}_p(\widetilde{Q})$, satisfying \eqref{eqn-211010-1130} - \eqref{eqn-211010-1136}. In the proof of \eqref{eqn-211010-1130}, we need to apply the parabolic poincar\'e type inequality for functions with zero boundary value.

Noting that
\begin{equation*}
	(\p_t-\Delta)(u\eta-w) = 0\,\,\text{in}\,\,\cQ_2,\quad u\eta-w = -w\,\,\text{on}\,\,\p_l\cQ_2,
\end{equation*}
and
\begin{equation*}
	u\eta=w=0\quad\text{on}\,\,\{t=-4\},
\end{equation*}
by the triangle inequality and the $\cL^1_p$ estimate for the initial-Dirichlet regularity problem in \cite[Theorem~5.23,~5.25]{B90}, we have
	\begin{equation}\label{eqn-211013-0806}
		\begin{split}
			\|\vec{N}_{1/2}(Du)\|_{L_p(\Delta_1)}
			&\leq
			\|\vec{N}_{1/2}(D(u\eta -w))\|_{L_p(\Delta_1)} + C\|Dw\|_{L_\infty(Q_{3/2})}
			\\&\leq
			C\|u\eta -w\|_{\cL^1_p(\p_l\cQ_2)} + C\|Dw\|_{L_\infty(Q_{3/2})}
			\\&=
			C\|w\|_{\cL^1_p(\p_l\cQ_2)} + C\|Dw\|_{L_\infty(Q_{3/2})}.
		\end{split}
\end{equation}
Now, fixing any $\epsi'\in (0,1-1/p)$, by the trace theorem, we have
\begin{equation}\label{eqn-211013-0805}
	\|w\|_{\cL^1_p(\p_l\cQ_2)} \leq C \|w\|_{B^p_{1+\epsi', par}(\p_l\cQ_2)} \leq C \|w\|_{B^p_{1+\epsi'+1/p, par}(\cQ_2)} \leq C \|w\|_{W^{1,2}_p(\cQ_2)}.
\end{equation}
Here $B^p_{\alpha, par}((S,T)\times\omega)$ are the parabolic Besov spaces with zero initial values. A straightforward definition is  through the Littlewood-Paley decomposition: Consider a usual cut-off $\varphi_0$ with $\supp(\varphi_0)\subset Q_2$ and $\varphi_0=1$ on $Q_1$. Then, let
\begin{equation*}
	\varphi_j(x,t) = \varphi_0(2^{-j}x, 2^{-2j}t) - \varphi_0(2^{-j+1}x, 2^{-2j+2}t).
\end{equation*}
From this, we define
\begin{equation*}
	B^p_{\alpha,par}(\bR\times\bR^d) = \Big\{f\in\mathcal{S}'(\bR\times\bR^d): \big(\sum_j 2^{j\alpha p}\|\mathcal{F}^{-1}(\varphi_j\mathcal{F}f)\|_{L_p(\bR\times\bR^d)}\big)^{1/p}\Big\},
\end{equation*}
where $\mathcal{F}$ is the Fourier transform in $(t,x)$. Finally, we can define $B^p_{\alpha, par}((S,T)\times\omega)$ by requiring $\supp(f)\cap\{t<S\}=\emptyset$, taking the extension in \eqref{eqn-211024-0621}, and localization. See for example, \cite[pp.~827-828]{MR2280778} and \cite[pp.~2000-2001]{MR3762093}.

Combining \eqref{eqn-211013-0806}, \eqref{eqn-211013-0805}, \eqref{eqn-211010-1130}, and \eqref{eqn-211010-1136}, the lemma is proved.
\end{proof}

\subsection{Weighted estimate and reverse H\"older inequality at the boundary}\label{sec-211031-1136}
Combining the above two lemmas, we have the following weighted estimate near $\Lambda$.
\begin{lemma}\label{lem-210622-1233}
Suppose that $r< R_0/2$, $A\ge 1$, and $u\in \cH^1_2$ satisfies \eqref{eqn-210827-0729}. Then for any $s\in \bR$ and any surface cube $\Delta_r$ with $\operatorname{dist}(\Delta_r,\Lambda)\leq Ar$,
	\begin{equation*}
		\int_{\Delta_r}|Du(Z)|^2\delta(Z)^{1-s}\,d\sigma
		\le C\int_{\Delta_{2r}}1_{\cN}|g_{\cN}|^2\delta(Z)^{1-s}\,d\sigma
		+C \int_{\cQ_{2r}}|Du(Z)|^2\delta(Z)^{-s}\,dZ,
	\end{equation*}
where $C=C(d,M, A, s)$.
	Here, the case $+\infty\leq +\infty$ is allowed.
\end{lemma}
The proof is by taking the Whitney decomposition of $\Delta_r\setminus \Gamma$ and applying Lemma \ref{lem-210216-0333-1} or \ref{lem-210216-0333-2} with $p=2$ on each Whitney cube. This is very similar to that of \cite[Lemma 4.9]{OB13} and thus omitted.

\begin{lemma}\label{lem-210706-0144}
	Let $\Delta_r$ be a surface cube as in Lemma \ref{lem-210622-1233}, $g_\cN\in L_\infty(\Delta_{2r})$, $p\in(1, p_0/2)$, where $p_0$ comes from Lemma \ref{lem-210903-0918} (a), and $u\in \cH^1_2$ satisfy \eqref{eqn-210827-0729}. Then there exists a sufficiently small $\epsi_0>0$ depending on $p_0$ and $p$ such that under Assumption \ref{ass-210730-0540} ($\epsi_0$), we have
	\begin{equation*}
			\big(\fint_{\Delta_r}|Du|^{p}\big)^{1/p}
			\leq
			\Big(\fint_{\Delta_r}|\vec{N}_{r/2}Du|^{p}\Big)^{1/p} \leq C \big(\fint_{\cQ_{2r}} |Du|^2\big)^{1/2} +C\|g_\cN\|_{L_\infty(\Delta_{2r})},
	\end{equation*}
where $C=C(d,M, p)$.
If furthermore $g_\cN=0$, we can take any $p\in(1,(m+2)/(m+1))$ provided that Assumption \ref{ass-210609-0500-3} ($\theta,m$) holds for a sufficiently small $\theta=\theta(d,M,p)$.
\end{lemma}
\begin{proof}
	By H\"older's inequality and Lemma \ref{lem-210622-1233} with $s$ to be chosen later,
	\begin{align}
		&\big(\fint_{\Delta_{r}}|Du|^p\big)^{1/p} \lesssim \big(\fint_{\Delta_{r}}|Du(Z)|^2\delta(Z)^{1-s}\big)^{1/2} \big(\fint_{\Delta_{r}} \delta(Z)^{(s-1)p/(2-p)}\big)^{1/p-1/2}\nonumber\\
		&\lesssim \big(\fint_{\Delta_{2r}}1_{\cN}|g_{\cN}|^2\delta(Z)^{1-s}\,d\sigma
		+ r\fint_{\cQ_{2r}}|Du(Z)|^2\delta(Z)^{-s}\,dZ\big)^{1/2}\big(\fint_{\Delta_{r}} \delta(Z)^{(s-1)p/(2-p)}\big)^{1/p-1/2}\notag
		\\&\lesssim
		\left(\|g_\cN\|_{L_\infty}\big(\fint_{\Delta_{2r}}\delta(Z)^{1-s}\big)^{1/2} + \big(\fint_{\cQ_{2r}}|Du|^{p_0}\big)^{1/p_0}\big(\fint_{\cQ_{2r}}\delta(Z)^{-s/(1-2/p_0)}\,dZ\big)^{1/2-1/p_0}r^{1/2}\right)\nonumber
		\\&\quad
		\cdot\big(\fint_{\Delta_{r}} \delta(Z)^{(s-1)p/(2-p)}\big)^{1/p-1/2}.\label{eqn-210828-0706}
	\end{align}
	By  Lemma \ref{lem-210903-0918} and H\"older's inequality,
	\begin{align*}
		\big(\fint_{\cQ_{2r}}|Du|^{p_0}\big)^{1/p_0} &\lesssim \fint_{\cQ_{4r}}|Du| + \big(\fint_{\Delta_{4r}}1_{\cN}|g_\cN|^{p_0(d+1)/(d+2)}\big)^{(d+2)/(p_0(d+1))}\\
		&\lesssim \big(\fint_{\cQ_{4r}}|Du|^2\big)^{1/2} + \|g_\cN\|_{L_\infty(\Delta_{4r})}.
	\end{align*}
Substituting back to \eqref{eqn-210828-0706}, we have
\begin{equation*}
	\begin{split}
		&\big(\fint_{\Delta_{r}}|Du|^p\big)^{1/p}\leq
		C\left(\big(\fint_{\cQ_{4r}}|Du|^2\big)^{1/2} + \|g_\cN\|_{L_\infty(\Delta_{4r})}\right)\\&\quad		\cdot\left(\big(\fint_{\Delta_{2r}}\delta(Z)^{1-s}\big)^{1/2} + \big(\fint_{\cQ_{2r}}\delta(Z)^{-s/(1-2/p_0)}\big)^{1/2-1/p_0}r^{1/2}\right) \big(\fint_{\Delta_{r}} \delta(Z)^{(s-1)p/(2-p)}\big)^{1/p-1/2}.
	\end{split}
\end{equation*}
We are left to cancel the integrals involving $\delta$. From Assumption \ref{ass-210730-0540} ($\epsi_0$), we need to choose $s$ such that
$$
1-s > -1+\epsi_0,\quad -s/(1-2/p_0)>-2+\epsi_0,\quad (s-1)p/(2-p)>-1+\epsi_0.
$$
Such choice is possible if
$$
1+(1-\varepsilon_0)(1-2/p)<(2-\varepsilon_0)(1-2/p_0),
$$
which holds provided that
$p<p_0/2$ and $\epsi_0$ is sufficiently small.


Now we have
\begin{equation}\label{eqn-210917-1002}
	\big(\fint_{\Delta_r}|Du|^{p}\big)^{1/p} \leq C \big(\fint_{\cQ_{2r}} |Du|^2\big)^{1/2} +C\|g_\cN\|_{L_\infty(\Delta_{2r})}.
\end{equation}
Next, we bound $\vec{N}_{r/2}(Du)$. By scaling and covering, we can simply derive from \eqref{eqn-210917-1002} that
\begin{equation}\label{eqn-210917-1009}
	\big(\fint_{\Delta_{7r/4}}|Du|^{p}\big)^{1/p} \leq C \big(\fint_{\cQ_{2r}} |Du|^2\big)^{1/2} +C\|g_\cN\|_{L_\infty(\Delta_{2r})}.
\end{equation}
Noting that
\begin{equation*}
	\Gamma_{r/2}(X)\subset \cQ_{3r/2}\quad\forall X\in\Delta_r,
\end{equation*}
Lemma \ref{lem-210216-0333-1} combined with a standard scaling and covering argument yield
\begin{align*}
	\big(\fint_{\Delta_r}|\vec{N}_{r/2}(Du)|^p\big)^{1/p} &\leq C\big(\fint_{\Delta_{7r/4}}|\p u/\p\vec{n}|^{p}\big)^{1/p} + C\big(\fint_{\cQ_{7r/4}} |Du|^p\big)^{1/p} \\
&\leq C \big(\fint_{\cQ_{2r}} |Du|^2\big)^{1/2} +C\|g_\cN\|_{L_\infty(\Delta_{2r})}.
\end{align*}
Here in the last inequality, we used \eqref{eqn-210917-1009} and H\"older's inequality. To prove the last assertion with $g_\cN=0$, we use Lemma \ref{lem-210903-0918} (b).
\end{proof}

\section{Existence of \texorpdfstring{$L_1$}{} solution on bounded Lipschitz domain}\label{sec-210713-0656}
\subsection{\texorpdfstring{$L_p$}{Lp} estimates for solutions with atom data}
Assume that $g_\cD=0$ and $g_{\cN}$ be an atom supported on $\cN^T\cap \Delta_r$ with $r<R_0/2$:
\begin{equation}\label{eqn-210510-1143}
	\norm{g_{\cN}}_{L_\infty}\leq \frac{1}{|\Delta_r|},\quad\fint_{\Delta_r}g_{\cN} = 0\,\,\text{when}\,\,\Delta_r\subset\cN.
\end{equation}
To facilitate the discussion, we solve \eqref{eqn-main} on a longer cylinder $\cQ^{T_1}$ with the height $T_1=T_1(T,R_0,\diam(\Omega))$ being sufficiently large. Since $g_\cN\in L_\infty\subset \cH^{-1/2}_2$, according to Theorem \ref{thm2.1}, the weak solution $u\in\cH^1_2(\cQ^T_1)$ to \eqref{eqn-main} exists as long as Assumption \ref{ass-0301-2356} holds. It is not difficult to see that corresponding to \eqref{eq8.21}, we have
	\begin{equation}
		\label{eqn-211029-1055}
		\|Du\|_{L_2(\cQ^{T_1})}\le C\|g_{\cN}\|_{L_{\frac{2(d+1)}{d+2}}(\cN^{T_1})},
	\end{equation}
where still $C=C(d, R_0, M, T, \diam(\Omega))$. Then we extend $u$ to be zero for $t<0$. This procedure will make sure $u$ is defined in all the (surface) cubes involved in this section.
	Taking a longer cylinder will not cause problem in the non-tangential maximal function estimate since it will increase the non-tangential cone.

Let us denote
$$\Sigma_k=\Delta_{2^k r}\setminus\Delta_{2^{k-1} r},\quad k\geq 2.$$
We aim to prove
\begin{proposition}\label{prop-210831-0431}
	Suppose that $\Lambda$ satisfies Assumptions \ref{ass-0301-2356} and \ref{ass-210809-0731}, and $u\in \cH^1_2(\cQ^{T_1})$ is the solution defined above with the atom data $g_{\cN}$ supported on $\Delta_r$. Let $p_0$ be the number in Lemma \ref{lem-210903-0918} (a) and $K$ be the positive integer such that $R_0/r \in (2^{K-1},2^K]$. Then there exists a constant $\beta=\beta(d,M)$
such that for any $p\in (1,p_0/2)$, we have
	\begin{equation}
		\label{eq8.09}
		\Big(\fint_{\Delta_{2r}}|Du|^p\,d\sigma\Big)^{1/p}\le C|\Delta_{r}|^{-1},
	\end{equation}
	\begin{equation}
		\label{eq8.10}
		\Big(\fint_{\Sigma_{k}}|Du|^p\,d\sigma\Big)^{1/p}\le C2^{-k\beta }|\Sigma_{k}|^{-1}\quad\forall k\in [2,K],
	\end{equation}
	and
	\begin{equation}\label{eqn-210831-0428}
		\norm{Du}_{L_p(\p_l\cQ^T\setminus\Delta_{2^Kr})} \leq C,
	\end{equation}
where $C=C(d,R_0,M,\diam(\Omega), T)$.
\end{proposition}

To prove Proposition \ref{prop-210831-0431}, we first prove the following decay estimate.
\begin{lemma}\label{lem-210706-0149}
	There exists a constant $\beta=\beta(d,M)\in(0,1)$, such that for any point $X\in ( -T_2,T+T_2)\times \Omega$ with
	$$
T_2=((1+\alpha)\diam(\Omega) + R_0)^2,\quad	R_X:=\dist(X,\Delta_r) >4r,
	$$
	we have
	\begin{equation*}
		\norm{Du}_{L_2(\cQ_{R_X/2}(X))} \leq C (r/R_X)^\beta R_X^{-d/2},
	\end{equation*}
where $C=C(d,R_0,M,\diam(\Omega), T)$.
\end{lemma}
\begin{proof}
 	We use the duality argument and the De Giorgi-Nash-Moser estimate following the argument in the proof of \cite[Lemma 2.2]{DK17}. Compared to the proofs of the elliptic case in \cite{OB13,TOB}, we do not use any Green's function estimate.
 	
 	More precisely, let $f\in C_0^\infty(\cQ_{R_X/2}(X))$ and $v\in \cH^1_2$ be the solution to the adjoint equation
	$$
	-v_t-\Delta v=\dv f
	$$
	with the zero terminal condition $v(T+T_2+R_X^2,\cdot)=0$ and zero Dirichlet and Neumann conditions on $\cD$ and $\cN$, respectively. By duality,
	\begin{equation}\label{eqn-210828-0744}
		\int_{\cQ} Du\cdot f=-\int_{\cN\cap \Delta_r} g_\cN v=-\int_{\Delta_r} 1_{\cN}g_\cN (v-v(X_0)),
	\end{equation}
	where $X_0\in \Delta_r$ can be any point when $\Delta_r\subset \cN$ and
	\begin{equation*}
		X_0\in \cD\cap\Delta_r\,\,\text{when}\,\,\Delta_r\cap \cD\neq\emptyset.
	\end{equation*}
In the last equality of \eqref{eqn-210828-0744}, we used the fact that $g_{\cN}$ has  zero mean when $\Delta_r\subset \cN$ and $v(X_0)=0$ when $\Delta_r\cap \cD\neq\emptyset$. Since $v$ satisfies the homogeneous heat equation in $\cQ_{R_X/2}$, by the De Giorgi-Nash-Moser estimate, we can find some $\beta=\beta(d,M)>0$, such that
	\begin{align}
		\sup_{Y\in\Delta_r}|v(Y)-v(X_0)|&\le (2\sqrt{d}r)^\beta[v]_{C^{\beta/2,\beta}(\cQ_{r})}
		\leq Cr^\beta R_X^{-\beta}(|v-v_0|^2)^{1/2}_{\cQ_{R_X/2}}\nonumber\\
		&\le Cr^\beta R_X^{-\beta}(|v-v_0|^2)^{1/2}_{\cQ_{R_X}}\nonumber\\
		&\le Cr^\beta R_X^{1-\beta}(|Dv|^2)^{1/2}_{\cQ_{R_X}}\label{eqn-210828-0757}\\
		&\le C r^\beta R_X^{1-\beta-(d+2)/2}\|f\|_{L_2(\cQ)},\label{eqn-210828-0801}
	\end{align}
where
\begin{equation*} v_0=(v)_{Q_{R_X}}\,\,\text{when}\,\,\cQ_{R_X/2}\cap\cD=\emptyset,\quad v_0=0\,\,\text{otherwise}.
\end{equation*}
In \eqref{eqn-210828-0757}, we applied the parabolic Poincar\'e inequality which can be found in \cite[Lemma~3.8]{CDL21} or \cite[Lemma~3.1]{MR2304157}: when $\cQ_{R_X/2}\cap\cD\neq\emptyset$,  Assumption \ref{ass-210809-0731} was needed to guarantee that $u=0$ on a large enough subset of the boundary; when $\cQ_{R_X/2}\cap\cD=\emptyset$, we applied the Poincar\'e inequality for functions with space-time zero mean.
Substituting \eqref{eqn-210828-0801} back into  \eqref{eqn-210828-0744}, noting $\|g_\cN\|_{L_1}\le 1$, the lemma is proved.
\end{proof}

\begin{proof}[Proof of Proposition \ref{prop-210831-0431}]
We first prove \eqref{eq8.09}. By \eqref{eqn-211029-1055} and the fact $\supp(g_\cN)\subset \Delta_r$,
\begin{equation}\label{eqn-210830-0644}
	\begin{split}
		\big(\fint_{\cQ_{4r}} |Du|^2\big)^{1/2} &\leq C|\cQ_{4r}|^{-1/2}\|Du\|_{L_2(\cQ^{T+5R_0^2})} \leq C|\cQ_{4r}|^{-1/2}\norm{g_\cN}_{L_{2(d+1)/(d+2)}(\cN^T)}
		\\&\leq C |\cQ_{4r}|^{-1/2}|\Delta_r|^{(d+2)/(2d+2)} \|g_\cN\|_{L_\infty(\Delta_r)}
		\\&\leq C/|\Delta_r|.
	\end{split}
\end{equation}

\textbf{Case 1}: $\operatorname{dist}(\Delta_{2r},\Lambda)> 2(\sqrt{d}+\sqrt 3)r$. In this case, $\Delta_{4r}\subset \cN$. We apply Lemma \ref{lem-210216-0333-1} with $p=2$, \eqref{eqn-210830-0644}, and \eqref{eqn-210510-1143} to obtain
\begin{align}\label{eqn-211017-0830-1}
	\big(\fint_{\Delta_{2r}} |Du|^2\big)^{1/2} \leq C \big(\fint_{\Delta_{4r}} |g_{\cN}|^2\big)^{1/2} + C\big(\fint_{\cQ_{4r}}|Du|^2\big)^{1/2} \le C/|\Delta_r|.
\end{align}

\textbf{Case 2}: $\operatorname{dist}(\Delta_{2r},\Lambda)\leq 2(\sqrt{d}+
\sqrt 3)r$. From Lemma \ref{lem-210706-0144}, \eqref{eqn-210830-0644},  and \eqref{eqn-210510-1143}, we have
	\begin{align}\label{eqn-211017-0830-2}
		\big(\fint_{\Delta_{2r}}|Du|^{p}\big)^{1/p}
		&\leq C
		\big(\fint_{\cQ_{4r}} |Du|^2\big)^{1/2} +C\|g_\cN\|_{L_\infty(\Delta_{4r})}
		\le C/|\Delta_r|.
	\end{align}
Hence, \eqref{eq8.09} is proved.

Next we prove \eqref{eq8.10}.
Take a covering $\Sigma_k\subset \bigcup\Delta_{2^{k-2}r}(Z)$ with only finite overlapping and \begin{equation}\label{eqn-211015-0705}
	\big(\Delta_{2^{k-1} r}(Z)\cap\supp(g_\cN)\subset\big)\Delta_{2^{k-1} r}(Z)\cap\Delta_r=\emptyset.
\end{equation}
On each $\Delta_{2^{k-2}r}(Z)$, there are three possibilities: $\Delta_{2^{k-1}r}(Z)\subset \cD$ (pure Dirichlet), $\Delta_{2^{k-1}r}(Z)\subset\cN$ (pure Neumann), or $\Delta_{2^{k-1}r}(Z)\cap\Lambda\neq\emptyset$ (mixed).

In the first two cases, we apply first Lemma \ref{lem-210216-0333-1} (pure Dirichlet) or Lemma \ref{lem-210216-0333-2} (pure Neumann, noting \eqref{eqn-211015-0705}) to obtain
\begin{equation}\label{eqn-210706-0931}
	\big(\fint_{\Delta_{2^{k-2}r}(Z)} |Du|^2\big)^{1/2} \leq C \big(\fint_{\cQ_{2^{k-1}r}(Z)}|Du|^2\big)^{1/2}\leq C2^{-k\beta}|\Sigma_k|^{-1}.
\end{equation}
Here in the last inequality, we applied Lemma \ref{lem-210706-0149}, noting $\dist(Z,\Delta_r)\approx 2^kr$ and $|\Sigma_k|\approx (2^kr)^{d+1}$.

When $\Delta_{2^{k-1}r}(Z)\cap\Lambda\neq\emptyset$, we apply Lemma \ref{lem-210706-0144} with $g_\cN=0$ on $\Delta_{2^{k-1}r}(Z)$ to obtain
\begin{equation}\label{eqn-210713-1259}
	\big(\fint_{\Delta_{2^{k-2}r}(Z)} |Du|^p\big)^{1/p} \leq C \big(\fint_{\cQ_{2^{k-1}r}(Z)}|Du|^2\big)^{1/2}\leq C2^{-k\beta}|\Sigma_k|^{-1}.
\end{equation}
Here, $2^{k-1}r \leq R_0$ was needed in order to apply Lemma \ref{lem-210706-0144}. Again, in the last inequality we applied Lemma \ref{lem-210706-0149}.

Adding up \eqref{eqn-210706-0931} and \eqref{eqn-210713-1259} for all cubes in the covering, we obtain \eqref{eq8.10}.

We are left to prove \eqref{eqn-210831-0428}. We cover $\p_l\cQ^T\setminus\Delta_{2^Kr}$ by surface cubes $\Delta_{R_0/4}(Z)$ with centers $Z$ satisfying $\text{dist}(Z,\Delta_r)\ge 2^{K-1}r$.
Similar to \eqref{eqn-210706-0931}, for surface cubes with $\Delta_{R_0/2}(Z)\subset \cD$ or $\cN$, we have
\begin{equation}\label{eqn-211017-0845-1}
	\begin{split}
		\big(\fint_{\Delta_{R_0/4}(Z)} |Du|^2\big)^{1/2} \leq C \big(\fint_{\cQ_{R_0/2}(Z)}|Du|^2\big)^{1/2}
		&\leq C 2^{-K\beta} (2^{K-1}r)^{-d/2}R_0^{-(d+2)/2}
		\\&\leq C 2^{-K\beta} R_0^{-d-1},
	\end{split}	
\end{equation}
 where in the last steps we used $2^{K-1}r\geq R_0/2$. Similarly, for surface cubes with $\Delta_{R_0/2}(Z)\cap\Lambda\neq\emptyset$, we have
\begin{equation}\label{eqn-211017-0845-2}
		\big(\fint_{\Delta_{R_0/4}(Z)} |Du|^p\big)^{1/p} \leq C \big(\fint_{\cQ_{R_0/2}(Z)}|Du|^2\big)^{1/2}
		\leq C 2^{-K\beta} R_0^{-d-1}.
\end{equation}
From these, \eqref{eqn-210831-0428} can be obtained by noting that $\p_l\cQ^T$ is bounded.
\end{proof}

\subsection{\texorpdfstring{$L_1$}{L1} estimate for the nontangential maximal function of \texorpdfstring{$Du$}{Du} and existence}\label{sec-211018-0507}
Now we are ready to prove Theorem \ref{thm-210830-0419} (a). As mentioned before, due to the $L_1$-solvability for the Dirichlet regularity problem in \cite{B90}, we may assume $g_\cD=0$.

Before giving our proof, let us mention that  some parts of the argument in \cite{OB13} do not work for the heat equation. In particular, for the heat equation, it seems not possible to find representation formulae as in the proof of \cite[Theorem 4.17]{OB13}, and we also do not have the equalities
$$
\int_{\partial\Omega}\partial u/\partial n\,d\sigma=0=\int_{\partial\Omega}\big(n_j\partial u/y_i-n_i\partial u/y_j)\,d\sigma
$$
used in \cite{TOB} and \cite{OB13}.
\begin{proof}[Proof of Theorem \ref{thm-210830-0419} (a)]
By approximation, we only prove that for any $H^1$ atom $g_{\cN}$ supported on $\cN^T\cap \Delta_{r}$ (see \eqref{eqn-210510-1143}), the weak solution $u\in\cH^1_2(\cQ^T)$ from Section \ref{sec1.1} satisfies the estimate \eqref{eqn-210910-0549}.
It is easily seen that Proposition \ref{prop-210831-0431} gives
$$
\int_{\p_l\cQ^T} |Du|\,d\sigma\le C.
$$
We only need to replace $|Du|$ above with $\vec{N}(Du)$.
In the following, let us also denote
$$\Sigma_1=\Delta_{2r}.$$
Recall the definition of $\Sigma_k$ and $K$ in Proposition \ref{prop-210831-0431}. Now on $\Sigma_k$ with $k\in[1,K]$, we first estimate $\vec{N}_{2^{k-3}r}(Du)$, for which we adapt the proof of Proposition \ref{prop-210831-0431}by replacing $Du$ with its the truncated non-tangential maximal function when applying the local estimates in Lemmas \ref{lem-210216-0333-1}, \ref{lem-210216-0333-2}, and \ref{lem-210706-0144}. More precisely, in the proof of Proposition \ref{prop-210831-0431}, in \eqref{eqn-211017-0830-1} and \eqref{eqn-211017-0830-2}, we replace
$$
\big(\fint_{\Delta_{2r}}|Du|^p\big)^p\quad\text{with}\,\,\big(\fint_{\Delta_{2r}}
\vec{N}_{r}(Du)^p\big)^{1/p},
$$
and in \eqref{eqn-210706-0931} and \eqref{eqn-210713-1259}, we replace
$$\big(\fint_{\Delta_{2^{k-2}r}(Z)}|Du|^p\big)^{1/p}\quad\text{with}\,\,\big(\fint_{\Delta_{2^{k-2}r}(Z)}\vec{N}_{2^{k-3}r}(Du)^p\big)^{1/p}.$$
Then the proof of Proposition \ref{prop-210831-0431} combined with H\"older's inequality yields
\begin{equation}\label{eqn-211017-0903-1}
	\int_{\Sigma_k}\vec{N}_{2^{k-3}r}(Du)\,d\sigma \leq \big(\fint_{\Sigma_k}(\vec{N}_{2^{k-3}r}(Du))^p\,d\sigma\big)^{1/p} \leq C2^{-k\beta}|\Sigma_k|^{-1}\quad \forall k\in[1,K].
\end{equation}
Now we estimate $\vec{N}^{2^{k-3}r}(Du)$. For any $X\in \Sigma_k$ with $k\in[1,K]$ and $Y\in \Gamma^{2^{k-3}r}(X)$, by definition of $\Gamma(X)$,
\begin{equation*}
	(\dist(Y,\Delta_r)\geq)d(Y)\geq \frac{|X-Y|}{1+\alpha} > \frac{2^{k-3}r}{1+\alpha}.
\end{equation*}
By the interior estimate for caloric functions,
\begin{equation*}
	|Du(Y)| \leq C(|Du|^2)^{1/2}_{\cQ_{\frac{2^{k-4}r}{1+\alpha}}(Y)} \leq C2^{-k\beta}|\Sigma_k|^{-1}.
\end{equation*}
Here, in the last inequality, for large $k$ satisfying $2^{k-3}r/(1+\alpha)>4r$, we applied Lemma \ref{lem-210706-0149}, noting that
\begin{equation*}
	\dist(Y,X) \leq (1+\alpha)\dist(Y,\p_l\cQ) \leq \diam(\Omega),
\end{equation*}
which implies $Y\in (0,T+T_2)\times\Omega$. For small $k$, such inequality can be simply derived from \eqref{eqn-211029-1055} and \eqref{eqn-210510-1143}.
Hence,
\begin{equation}\label{eqn-211017-0903-2}
	\|\vec{N}^{2^{k-3}r}(Du)\|_{L_\infty(\Sigma_k)} \leq C2^{-k\beta}|\Sigma_k|^{-1} \quad \forall k\in[1,K].
\end{equation}
Combining \eqref{eqn-211017-0903-1} and \eqref{eqn-211017-0903-2}, we have: for $k\in[1,K]$,
\begin{equation}\label{eqn-211017-0914}
	\int_{\Sigma_k}\vec{N}(Du) \leq C2^{-k\beta}|\Sigma_k|^{-1}.
\end{equation}
In the same spirit with \eqref{eqn-211017-0903-1}, we can adapt the argument in \eqref{eqn-211017-0845-1} and \eqref{eqn-211017-0845-2} to obtain: for any $Z$ with $\dist(Z,\Delta_r)\geq 2^{K-1}r$,
\begin{equation}\label{eqn-211017-0912-1}
	\fint_{\Delta_{R_0/2}(Z)}\vec{N}_{R_0/4}(Du) \leq \big(\fint_{\Delta_{R_0/2}(Z)}(\vec{N}_{R_0/4}(Du))^p\big)^{1/p} \leq C2^{-K\beta}R_0^{-d-1}.
\end{equation}
Furthermore, similar to \eqref{eqn-211017-0903-2}, by the interior estimate and Lemma \ref{lem-210706-0149}, we have
\begin{equation}\label{eqn-211017-0912-2}
	\|\vec{N}^{R_0/4}(Du)\|_{L_\infty(\Delta_{R_0/2}(Z))} \leq C(|Du|^2)^{1/2}_{\cQ_{\frac{R_0}{4(1+\alpha)}}(Y)} \leq C2^{-K\beta}R_0^{-d-1}.
\end{equation}
Hence, from \eqref{eqn-211017-0912-1} and \eqref{eqn-211017-0912-2}, we have
\begin{equation}\label{eqn-211017-0915}
	\int_{\Delta_{R_0/2}(Z)} \vec{N}(Du) \leq C2^{-K\beta}R_0^{-d-1}.
\end{equation}
Finally, as the proof of Proposition \ref{prop-210831-0431}, we take the summation of \eqref{eqn-211017-0914} for $k\in[1,K]$, then cover the rest of $\p_l\cQ^T$ with finitely many cubes with radii $R_0/2$, and apply \eqref{eqn-211017-0915} on each of them to obtain the desired estimate. Hence, Theorem \ref{thm-210830-0419} (a) is proved.
\end{proof}

\section{Uniqueness}\label{sec-210817-0523}
In this section, we prove the uniqueness in Theorem \ref{thm-210830-0419} (b). More precisely, suppose that $u\in \cH^1_{2,\text{loc}}(\cQ^T)$ solves
\begin{equation*}
	\begin{cases}
		u_t-\Delta u = 0  & \text{in }\, \cQ^{T},\\
		\frac{\p u}{\p \vec{n}} = 0  & \text{on }\, \cN^{T},\\
		u = 0 & \text{on }\, \cD^{T},\\
		u = 0 & \text{on }\, \{0\}\times\Omega,\\
		\vec{N}(Du) \in L_q(\p_l\cQ^{T})
	\end{cases}
\end{equation*}
for some $q\geq 1$, then $u=0$. Recall that by a solution, we mean that the weak formulation holds for any test function $\Psi\in C^\infty_c(\cQ^{T})$, with the boundary conditions being satisfied in the sense of the ``non-tangential limit''. In the following, we focus on the case when $q=1$, which will imply the uniqueness when $q>1$.

Fix a sufficiently small number $h_0$. Let $C_\alpha$ be a constant to be chosen later such that \eqref{eqn-211101-0708} holds for any $h\in(0,h_0)$. Such choice can be made independent of $h_0$. To prove $u\equiv 0$, we show that for any small cylinder $Q\subset\cQ^{T-(C_\alpha h_0)^2}$,
\begin{equation}\label{eqn-210809-0715}
	\int_\cQ u1_Q = 0.
\end{equation}
For this, we solve the dual problem with zero terminal value
\begin{equation}\label{eqn-210809-0453}
	\begin{cases}
		v_t + \Delta v = 1_Q  & \text{in }\, \cQ^{T},\\
		\frac{\p v}{\p \vec{n}} = 0  & \text{on }\, \cN^T,\\
		v = 0 & \text{on }\, \cD^T,\\
		v = 0 & \text{on }\, \{t=T\}\times\Omega.
	\end{cases}
\end{equation}
Since $\Lambda$ satisfies Assumption \ref{ass-0301-2356}, the existence of a solution $v\in\cH^1_2(\cQ^{T})$ is guaranteed by \cite[Proposition 3.3]{CDL21}. By the uniqueness of weak solutions,
	\begin{equation*}
		v=0\quad\text{for}\,\,t\geq T-(C_\alpha h_0)^2.
	\end{equation*}
Furthermore, by the De Giorgi-Nash-Moser estimate, there exists some $\beta\in(0,1)$ such that $v\in C^{\beta/2,\beta}(\cQ^T\cup\p_l\cQ^T)$. Noting that till this point,  only Assumptions \ref{ass-0301-2356} and \ref{ass-210809-0731} are needed.

The major difficulty here is that we cannot directly apply $u$ as a test function for \eqref{eqn-210809-0453}, since it is not regular enough near $\p_l\cQ$. For this, we ``regularize'' $u$ by taking a cut-off near $\cD$ and translating near $\cN$. For this, we introduce a cut-off function in the following lemma.
\begin{lemma}\label{lem-211105-0444}
	Let $h\in(0,h_0)$, $S\in(0,\infty)$, and $\cQ^S$, $\cD^S$, $\cN^S$, and $\Lambda^S$ be the truncated sets defined as in \eqref{eqn-211105-0426}.
 There exists a function $\eta\in C^{1/2,1}_{loc}(\bR^{1+d}\setminus \overline{\Lambda^S})\cap C^\infty(\bR^{1+d}\setminus \overline{\p_l\cQ^S})$, such that
 \begin{enumerate}
 	\item $\overline{\supp(\eta)}\cap \cD^S=\emptyset$. Outside the $2h$-neighborhood of $\cD^S$, $\eta=1$. On $\cN^S$, $\eta=1$.
 	\item For $\delta_S(X):=\dist(X,\Lambda^S)$:
 	\begin{align}
 		&d(X)\approx \delta_S(X)\quad \forall X\in\supp(D\eta)\cap\{\delta_S(X)<h\}\cap\cQ^S,\label{eqn-210927-1003} \\
 		&d(X)\approx \operatorname{dist}(X,\cD^S)\quad \forall X\in \supp(D\eta)\cap\cQ^S.\label{eqn-210925-0429}
 	\end{align}
 	\item For any $X\in\supp(D\eta)\cap\cQ^S$,
 	\begin{equation}
 		\label{eq3.53}
 		\begin{split}
 			&|D\eta(X)|\leq C\max\{1/h, 1/\delta_S(X)\},
 			\\
 			&|D^2\eta(X)| + |\p_t\eta(D\eta)|\leq C\max\{1/h^2, 1/\delta_S(X)^2\}.
 		\end{split}	
 	\end{equation}
 \end{enumerate}
\end{lemma}

\begin{figure}
	\begin{center}
\begin{tikzpicture}
	\draw [<->,thick] (0,3) node (yaxis) [above] {$x^1$} |- (8,0) node (xaxis) [right] {$x'$};
	\draw plot [smooth] coordinates { (5,0) (4,1) (0,1)};
	\draw plot [smooth] coordinates { (5,0) (6.4,1.4) (6,2) (0,2)};
	\node at (2,0.5) {$\eta=0$};
	\node at (4,2.5) {$\eta=1$};
	\node at (3,1.5) {$\supp(D\eta)$};
	\node[left] at (0,1) {$h$};
	\node[left] at (0,2) {$2h$};
	\node[below] at (3,0) {$\cD$};
	\node[below] at (5,0) {$\Lambda$};
	\node[below] at (7,0) {$\cN$};
\end{tikzpicture}
\caption{Cut-off $\eta$ for a fixed $t$} \label{cut-off figure}
\end{center}
\end{figure}
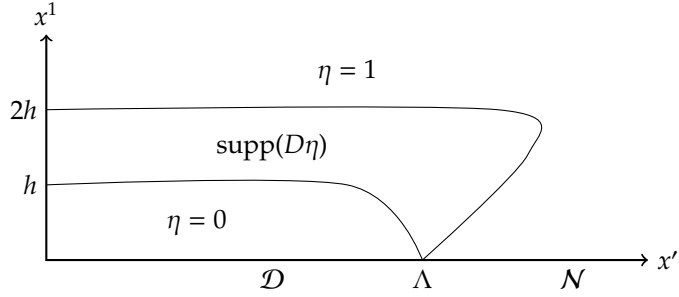
See Figure \ref{cut-off figure}. The proof of Lemma \ref{lem-211105-0444} is by using the regularized distance, which can be found in Appendix \ref{app-210824-1212}. In the following, we take the function $\eta$ in Lemma \ref{lem-211105-0444} with $S=T-(C_\alpha h_0)^2$. For simplicity, we denote
\begin{equation*}
	\delta_{T,h_0}(X):=\delta_{T-(C_\alpha h_0)^2}(X) (=\dist(X,\Lambda^{T-(C_\alpha h_0)^2}))\quad\text{and}\,\,\cQ^{T,h_0}:=\cQ^{T-(C_\alpha h_0)^2}.
\end{equation*}
Similarly, we denote $\cD^{T,h_0}, \cN^{T,h_0}$, and $\Lambda^{T,h_0}$.

Next, we take a partition of unity
\begin{equation*}
	1_\Omega = \sum_{k=0}^l \chi_k(x),
\end{equation*}
such that $\dist(\supp(\chi_0),\p\Omega)>0$, and for each $k\geq 1$, after a rotation,
\begin{equation}\label{eqn-210810-1209}
	\Omega\cap\supp(\chi_k) = \{x^1>\psi(x')\}\cap\supp(\chi_k),
\end{equation}
where $\psi$ is a Lipschitz function.
Now for each $k\geq 1$, we define the translation
\begin{equation*}
	\tau_h^{(k)}(t, x^1,x^2,\ldots,x^d) = (t,x^1+h,x^2,\ldots,x^d).
\end{equation*}
Then, for $h\in(0,h_0)$, $$w:=\big((u\chi_k)\circ\tau_h^{(k)}\big)\eta \in C^\infty((\cQ^{T,h_0}\cup\p_l\cQ^{T,h_0})\setminus\Lambda^{T,h_0}).$$
Furthermore, $w$ satisfies
\begin{equation*}
	\begin{cases}
		(\p_t - \Delta) w = \big((\p_t-\Delta)(u\chi_k)\big)\circ\tau_h^{(k)} \eta + f_h^{(k)} & \text{in }\, \cQ^{T,h_0},\\
		w = 0 & \text{on }\, \cD^{T,h_0},\\
		w = 0 & \text{on }\, \{t=0\}\times\Omega,
	\end{cases}
\end{equation*}
where
\begin{equation*}
	f_h^{(k)} = (u\chi_k)\circ\tau_h^{(k)}\p_t\eta - 2\big(\nabla(u\chi_k)\big)\circ\tau_h^{(k)} \cdot\nabla \eta -(u\chi_k)\circ\tau_h^{(k)}\Delta\eta.
\end{equation*}
We will prove
\begin{align}
	\int_{\cQ^T} w 1_Q
	&=
	-
	\int_{\cQ^T} \big(2(\nabla u\cdot\nabla \chi_k)\circ\tau_h^{(k)}\eta v + (u\Delta\chi_k)\circ\tau_h^{(k)}\eta v\big)
	+
	\int_{\cQ^T} f_h^{(k)} v
	+
	\int_{\cN^T}\frac{\p w}{\p\vec{n}}v \label{eqn-210810-1234}
	\\
&:={\rm I}+{\rm II}+{\rm III}\notag\\
&
	\rightarrow  -\int_{\cQ^T} (2\nabla u\cdot\nabla \chi_k + u\Delta\chi_k)v + \int_{\cN^T} \frac{\p\chi_k}{\p\vec{n}} uv,\quad\text{as}\,\,h\rightarrow 0.\label{eqn-210817-0618}
\end{align}
Note that the integrals above can be replaced with the integrals on $\cQ^{T,h_0}$ since $1_Q$ and $v$ vanishes in $\cQ^T\setminus \cQ^{T,h_0}$.
In the following, we will focus on proving \eqref{eqn-210817-0618}, during which the integrability of the right-hand side of \eqref{eqn-210810-1234} will be proved. From this, \eqref{eqn-210810-1234} can be proved by taking usual cut-off functions which are zero near $\Lambda$.

We prove \eqref{eqn-210817-0618} by computing term by term. More precisely, we show as $h\rightarrow 0$,
\begin{align}
	{\rm I} & \rightarrow -\int_{\cQ^T} (2\nabla u\cdot\nabla \chi_k + u\Delta\chi_k)v,\label{eqn-210817-0614-1}
	\\
	{\rm II} & \rightarrow 0,\label{eqn-210817-0614-2}
	\\
	{\rm III} & \rightarrow\int_{\cN^T} \frac{\p\chi_k}{\p\vec{n}} uv.\label{eqn-210817-0614-3}
\end{align}
To make the limit
\begin{equation*}
		\tau_h^{(k)}(X) \rightarrow X\in\p_l\cQ^T
\end{equation*}
to be ``non-tangential'', we choose $\alpha$ (the aperture of non-tangential cones) large enough such that for any $X=(t,x)$ with $x \in \supp(\chi_k)\cap\p\Omega$, and $t\in(0,T)$, under the coordinate system in \eqref{eqn-210810-1209},
\begin{equation}\label{eqn-210816-0508}
	 (t,y^1,x')\in\Gamma(X)\quad\forall y^1\in(\psi(x'),\psi(x')+2h).
\end{equation}
As a preparation, we first prove that $\vec{N}(u)\in L_1(\p_l\cQ^T)$ and
	\begin{equation}\label{eqn-210920-1223}
		\|\vec{N}(u)\|_{L_1(\p_l\cQ^T)}\leq C\|\vec{N}(Du)\|_{L_1(\p_l\cQ^T)}.
	\end{equation}
Noting that the non-tangential limit of $u$ exists at $\p_l\cQ^T$ and $u=0$ on $\cD^T$, by the Sobolev-Poincar\'e inequality on $\p_l\cQ^T$ and H\"older's inequality,
\begin{equation}\label{eqn-210920-1218}
	\norm{u}_{L_1(\p_l\cQ^T)} \leq C\norm{u}_{L_{(d+1)/d}(\p_l\cQ^T)} \leq C\norm{Du}_{L_1(\p_l\cQ^T)} \leq C\norm{\vec{N}(Du)}_{L_1(\p_l\cQ^T)}.
\end{equation}
By the fundamental theorem of calculus, for any $X\in\p_l\cQ$ and $Y\in\Gamma(X)$, we have
\begin{equation*}
	|u(Y)-u(X)| \leq C\vec{N}(Du)(X),
\end{equation*}
from which,
\begin{equation*}
	\vec{N}(u)(X) \leq |u(X)| + C\vec{N}(Du)(X).
\end{equation*}
Integrating in $X\in\p_l\cQ^T$ and noting \eqref{eqn-210920-1218}, we reach \eqref{eqn-210920-1223}.

Now we prove \eqref{eqn-210817-0614-1}. Noting \eqref{eqn-210816-0508} and $\chi_k\in C^\infty$, we have
\begin{align*}
	&\sup_h\Big(\big|(\nabla u\cdot\nabla\chi_k)\circ\tau_h^{(k)}(X)\big| + \big|(u\Delta\chi_k)\circ\tau_h^{(k)}(X)\big|\Big) \\
&\leq C(|\vec{N}(Du)(t,\psi(x'),x')| + |\vec{N}(u)(t,\psi(x'),x')|),
\end{align*}
which implies
\begin{equation}\label{eqn-211114-1105}
	\sup_h\Big(\big|(\nabla u\cdot\nabla\chi_k)\circ\tau_h^{(k)}(X)\big| + \big|(u\Delta\chi_k)\circ\tau_h^{(k)}(X)\big|\Big) \in L_1^{t,x^2,\ldots x^d}L_\infty^{x^1}.
\end{equation}
Also noting $0\leq \eta\uparrow 1$ as $h\to 0$ and $v\in L_\infty(\cQ^T\cup\p_l\cQ^T)$, by the dominated convergence theorem, \eqref{eqn-210817-0614-1} can be obtained by passing $h\rightarrow 0$ under the integration.

Similarly, since $\eta=1$ on $\cN^{T,h_0}$ and in view of \eqref{eqn-210816-0508},
\begin{align}
	\lim_{h\rightarrow 0} {\rm III}
	&=
	\lim_{h\rightarrow 0}\int_{\cN^T} \frac{\p((u \chi_k)\circ\tau_h^{(k)})}{\p\vec{n}}\eta v
	=
	\lim_{h\rightarrow 0}\int_{\cN^T} \big(\frac{\p (u \chi_k)}{\p\vec{n}}\big)\circ\tau_h^{(k)}\eta v\nonumber
	\\&= \int_{\cN^T} \lim_{h\rightarrow 0} \big(\frac{\p(u \chi_k)}{\p\vec{n}}\big)\circ\tau_h^{(k)}\eta v\label{eqn-210809-0651}
	\\&= \int_{\cN^T} \frac{\p\chi_k}{\p\vec{n}} uv,\nonumber
\end{align}
which proves \eqref{eqn-210817-0614-3}.
Here $\vec{n}$ is the unit outer normal at the boundary, extended  in a parallel way into the domain, i.e.,
\begin{equation*}
	\vec{n}\vert_{(t,x^1,x')} = \vec{n}\vert_{(t,\psi(x'),x')}.
\end{equation*}
In \eqref{eqn-210809-0651} we interchange the integration and the limit due to the dominated convergence theorem, noting
\begin{equation*}
	\sup_h\Big|\frac{\p (u\chi_k)}{\p\vec{n}}\circ\tau_h^{(k)}\Big|\in L_1(\cN^T),\quad 0\leq \eta\uparrow 1,\quad\text{and}\,\,v\in L_\infty(\overline{\cQ^T}).
\end{equation*}
We are left to show \eqref{eqn-210817-0614-2}.
Since $v=0$ on $\cD^{T}$, by the De Giorgi-Nash-Moser estimate, for any $X\in \cQ^T$, we have
\begin{equation*}
	|v(X)| \leq C\dist(X,\cD^{T})^\beta \leq C\dist(X,\cD^{T,h_0})^\beta.
\end{equation*}
By \eqref{eqn-210927-1003}, \eqref{eqn-210925-0429}, and \eqref{eq3.53} with $S=T-(C_\alpha h_0)^2$,
\begin{align}
	|(v\nabla\eta)(X)| \leq C \dist(X,\cD^{T,h_0})^{\beta-1} \approx C d(X)^{\beta-1}\quad\forall X\in \cQ^T.
                    \label{eqn-210815-1213}
\end{align}
From this, noting $\chi_k\in C^\infty$, by the Minkowski inequality, we have
\begin{align}
	&\big|\int_{\cQ^T}  - 2\nabla(u\chi_k)\circ\tau_h^{(k)}\cdot\nabla \eta v\big|\notag\\
	&\leq
	C \|\nabla(u\chi_k)\circ\tau_h^{(k)}\|_{L_1^{t,x^2,\ldots,x^d}L_\infty^{x^1}} \|v\nabla\eta\|_{L_1^{x^1}L_\infty^{t,x^2,\ldots,x^d}} \nonumber
	\\&\leq
	C\big(\|\vec{N}(Du)\|_{L_1(\p_l\cQ^T)} + \|\vec{N}(u)\|_{L_1(\p_l\cQ^T)}\big)\int_0^{2h} r^{\beta-1}\,dr\label{eqn-21081-0721}
	\\&\leq
	C\big(\|\vec{N}(Du)\|_{L_1(\p_l\cQ^T)} + \|\vec{N}(u)\|_{L_1(\p_l\cQ^T)}\big) h^\beta \rightarrow 0.\notag
\end{align}
Here in \eqref{eqn-21081-0721}, we used \eqref{eqn-210816-0508} and \eqref{eqn-210815-1213} combined with the fact that for $X=(t,x^1,\ldots,x^d)$ with $x\in\supp(\chi_k)$,
$$d(X) \approx x^1-\psi(x^2,\ldots,x^d))\in(0,2h)\quad\text{on}\,\,\supp(D\eta)\cap\cQ^{T,h_0}.$$

Estimating terms involving $\p_t\eta$ and $\Delta\eta$ requires more dedicate estimates near $\Lambda$. In the sequel, let us denote
\begin{equation*}
	\Lambda^{T,h_0}_h:=\{X\in\cQ^{T,h_0}:\dist(X, \Lambda^{T,h_0})<h\}.
\end{equation*}
Similar to \eqref{eqn-210815-1213}, using \eqref{eqn-210927-1003} and \eqref{eq3.53},
we have
\begin{equation}\label{eqn-210817-0819}
	|(v \p_t\eta)(X)| + |(v D^2\eta)(X)| \leq Ch^{\beta-2}\quad\text{on}\,\,(\supp(D\eta)\setminus\Lambda^{T,h_0}_h)\cap\cQ^{T,h_0}
\end{equation}
and
\begin{equation}\label{eqn-210817-0815}
	|(v \p_t\eta)(X)| + |(v D^2\eta)(X)| \leq C\delta_{T, h_0}(X)^{\beta-2}\quad\text{on}\,\,(\supp(D\eta)\cap\Lambda^{T,h_0}_h)\cap\cQ^{T,h_0}.
\end{equation}
For any point
\begin{equation}\label{eqn-211017-1002-1}
	X=(t,x)\in\supp(D\eta)\cap\Lambda^{T,h_0}_h \,\, \text{with}\,\,x\in\supp(\chi_k),
\end{equation}
denote
\begin{equation}\label{eqn-211017-1002-2}
		\widetilde{X}\in\Lambda^{T,h_0}\,\,\text{to be the point with}\,\,|X-\widetilde X| =  \delta_{T, h_0}(X).
	\end{equation}
Now we denote
$$
\widetilde \cQ^T=\cQ\cap\{t\in (-T,T)\},\quad
\widetilde \cD^T=\cD\cap\{t\in (-T,T)\},\quad
\widetilde \cN^T=\cN\cap\{t\in (-T,T)\},
$$
and take the zero extension of $u$ for $t\in (-T,0)$. We redefine the non-tangential maximal function by using parabolic non-tangential cones restricted to $\widetilde \cQ^T$. It is easily seen that for $t\ge 0$, $\vec{N}(Du)(t,\cdot)$ does not change values and for $t<0$,
\begin{equation}
                    \label{eq3.55}
\vec{N}(Du)(t,\cdot)\le \vec{N}(Du)(-t,\cdot).
\end{equation}
By Assumption \ref{ass-210809-0731} and the fact $\Delta_{C_\alpha h}(\widetilde{X})\subset\p_l\widetilde\cQ^T$, we have
\begin{equation}\label{eqn-211017-0959}
	|\Delta_{C_\alpha h}(\widetilde{X})\cap\widetilde\cD^T| \geq Ch^{d+1}.
\end{equation}
Similar to \eqref{eqn-210816-0508},  also noting $\delta_{T, h_0}(X)\approx d(X)$, we can make
$X\,\,\text{and}\,\, \tau_h^{(k)}(X)\in\Gamma(\widetilde{X})$
by choosing $\alpha$ large enough.
Furthermore, noting that
\begin{equation*}
	d(\tau_h^{(k)}(X))\approx |\tau_h^{(k)}(X)-\widetilde{X}|\approx h,
\end{equation*}
by the triangle inequality, there exists some (small) constant $C_\alpha$ such that
\begin{equation}\label{eqn-211101-0708}
	\tau_h^{(k)}(X)\in\Gamma(Y)\quad\forall Y\in\Delta_{C_\alpha h}(\widetilde{X}).
\end{equation}
Since $u=0$ on $\cD^T$, 
by the fundamental theorem of calculus, \eqref{eqn-211017-0959}, and \eqref{eq3.55},
\begin{equation}\label{eqn-210816-0717}
	(|u|\chi_k)\circ\tau_h^{(k)}(X)
	\leq
	Ch\fint_{\Delta_{C_\alpha h}(\widetilde{X})\cap\widetilde\cD^T}|\vec{N}(Du)|
\le Ch^{-d} \int_{\Delta_{C_\alpha h}(\widetilde{X})\cap\p_l\cQ^T}|\vec{N}(Du)|.
\end{equation}
By taking a large number $C$ (for example, $C=\max\{10C_\alpha,1\}$), we cover $\Lambda^{T,h_0}$ with finitely many surface cubes $\Delta_{Ch}(X_i)$, such that for any $X$ and $\widetilde{X}$  in \eqref{eqn-211017-1002-1} and \eqref{eqn-211017-1002-2}, we have $\Delta_{C_\alpha h}(\widetilde{X})\subset \Delta_{Ch}(X_i)$ for some $i$, and $|X_i-X_j|\ge Ch/2$ for different $i$ and $j$. From \eqref{eqn-210816-0717}, \eqref{eqn-210817-0815}, and the fact $\delta(X)\leq\delta_{T, h_0}(X)$, we have
\begin{align}
	\big|\int_{\cQ^T \cap\Lambda^{T, h_0}_h} (u\chi_k)\circ\tau_h^{(k)}\Delta\eta v\big|
	&\leq
	\sum_i Ch^{-d}\int_{\Delta_{Ch}(X_i)\cap\p_l\cQ^T}\vec{N}(Du)\int_{\cQ_{Ch}(X_i)
\cap\cQ^T}\delta_{T, h_0}(X)^{\beta-2} \nonumber
	\\&\leq
	Ch^{-d}\norm{\vec{N}(Du)}_{L_1(\p_l\cQ^T)} \sup_i \int_{\cQ_{Ch}(X_i)}\delta(X)^{\beta-2} \nonumber
	\\&\le
	Ch^{-d}\norm{\vec{N}(Du)}_{L_1(\p_l\cQ^T)} h^{d+2+\beta-2}\label{eqn-210920-0603}
	\\&\le
	Ch^\beta \norm{\vec{N}(Du)}_{L_1(\p_l\cQ^T)}\rightarrow 0.\label{eqn-210817-0130-1}
\end{align}
Here, Assumption \ref{ass-210730-0540} ($\epsi_0$) with $\epsi_0<\beta$ was needed in \eqref{eqn-210920-0603}.
The estimate away from $\Lambda$ is simpler: from \eqref{eqn-210817-0819},
\begin{align}
	\big|\int_{\cQ^T\setminus\Lambda^{T , h_0}_h} (u\chi_k)\circ\tau_h^{(k)}\Delta\eta v\big|
	&\leq
	\|\vec{N}(Du)\|_{L_1(\p_l\cQ^T)}\int_0^h C(x^1+h)h^{\beta-2}\,dx^1\nonumber
	\\&\leq
	C\norm{\vec{N}(Du)}_{L_1(\p_l\cQ^T)}h^{\beta}\rightarrow 0.\label{eqn-210817-0930}
\end{align}

From \eqref{eqn-210817-0130-1} and \eqref{eqn-210817-0930},
\begin{equation*}
	\int_{\cQ^T} (u\chi_k)\circ\tau_h^{(k)}\Delta\eta v \rightarrow 0.
\end{equation*}
By the same reasoning,
\begin{equation*}
	\int_{\cQ^T} (u\chi_k)\circ\tau_h^{(k)}\p_t\eta v \rightarrow 0.
\end{equation*}
Hence \eqref{eqn-210817-0614-2}, and furthermore \eqref{eqn-210810-1234}-\eqref{eqn-210817-0618} are proved.
We turn to the proof of \eqref{eqn-210809-0715}. For this, since similar to \eqref{eqn-211114-1105},
\begin{equation*}
	\sup_h|(u\chi_k)\circ\tau_h^{(k)}\eta 1_Q|\in L^{t,x^2,\ldots,x^d}_1L^{x^1}_\infty,
\end{equation*}
we can apply the dominated convergence theorem and \eqref{eqn-210810-1234}-\eqref{eqn-210817-0618} to obtain: for any $k\in\{1,\ldots,l\}$,
\begin{align}
	\int_{\cQ^T} u\chi_k 1_Q
	&=
	\lim_{h\rightarrow 0}\int_{\cQ^T} (u\chi_k)\circ\tau_h^{(k)}\eta 1_Q \nonumber
	\\&=
	-\int_{\cQ^T} (2\nabla u\cdot\nabla \chi_k + u\Delta\chi_k)v  +\int_{\cN^T} \frac{\p\chi_k}{\p\vec{n}} uv.\label{eqn-210810-1233}
\end{align}
For $k=0$, since $u\chi_0\in C_c^\infty([0,T]\times \Omega)$ and $u\chi_0=0$ when $t=0$, we can test \eqref{eqn-210809-0453} by $u\chi_0$ to get
\begin{equation}\label{eqn-210810-1236}
	\int_{\cQ^T} u\chi_0 1_Q = -\int_{\cQ^T} (2\nabla u\cdot\nabla \chi_0 + u\Delta\chi_0)v.
\end{equation}
Adding up \eqref{eqn-210810-1233} and \eqref{eqn-210810-1236}, noting $\sum_k \chi_k=1$, we reach \eqref{eqn-210809-0715}. Since $h_0$ can be arbitrarily small, $u$ has to be zero, which proves the uniqueness.

\section{Higher regularity}\label{sec-211018-0509}

\subsection{\texorpdfstring{$L_{q}$}{Lq} estimate for the nontangential maximal function of \texorpdfstring{$Du$}{Du}}\label{sec-210901-0632}
Under additional assumptions on $\p\cQ$ and $\Lambda$ as in \cite{CDL21}, we can obtain the $L_q$ estimate when $q\in (1,\frac{m+2}{m+1})$. In this section, we prove Theorem \ref{thm-211009-0913}.

The following interpolation lemma is a parabolic  analog of \cite[Theorem~3.2]{Shen}, which in turn is in spirit of a result in \cite{CP}. As in \cite{TOB} and \cite{OB13}, compared to the original version in \cite[Theorem~3.2]{Shen}, here we need to replace $|f|$ on the right-hand sides with $|f|^s$ for some $s>1$ since we start from the $L_1$-solvability with data in the Hardy space.

In this section, for a surface cube $\Delta=\Delta_r(X)$ and $l>0$, we denote $l\Delta=\Delta_{lr}(X)$ to be a dilated cube.
\begin{lemma}\label{lem-210827-0402}
	Let $\Delta_0$ be a (parabolic) surface cube such that $\text{diam}(\Delta)< R_0$, $s,q$ be two  numbers satisfying $1<s<q$, and $F, g$ be two functions defined on $\Delta_0$. Suppose that for any surface cube $\Delta\subset \Delta_0/4$, we can find $F_\Delta$ and $R_\Delta$ defined on it, such that
	\begin{align}
		|F| &\leq C_0\Big(|F_\Delta| + |R_\Delta|\Big)\quad\text{on}\,\,\Delta,\label{eqn-210827-0358-1}
		\\
		\big(\fint_{\Delta} |R_\Delta|^q\big)^{1/q} &\leq C_1 \left(\fint_{16\Delta} |F| + \sup_{\Delta'\supset 4\Delta}\big(\fint_{\Delta'}|f|^{s}\big)^{1/s} \right),\label{eqn-210827-0358-2}
		\\
		\fint_\Delta |F_\Delta| &\leq C_2 \sup_{\Delta'\supset 4\Delta}\big(\fint_{\Delta'} |f|^{s}\big)^{1/s},\label{eqn-210827-0358-3}
	\end{align}
where $C_0,C_1,C_2\ge 0$ are constants.
Then for any $p\in(s,q)$, if $f\in L_p(\Delta_0)$, then we also have $F\in L_p(\Delta_0/4)$ with
\begin{equation*}
	\big(\fint_{\Delta_0/4} |F|^p\big)^{1/p} \leq C\left(\fint_{\Delta_0} |F| + \big(\fint_{\Delta_0}|f|^p\big)^{1/p}\right).
\end{equation*}
Here $C=C(d,M,s,p,q,C_0,C_1,C_2)$.
\end{lemma}

Now let $s\in (1,q)$ and suppose that $u$ is the solution to \eqref{eqn-main} with $g_\cD=0$, $g_\cN\in L_p(\cN^T)\subset L_s(\cN^T)\subset H^1(\cN^T)$, and $\vec{N}(Du)\in L_1$ obtained in Theorem \ref{thm-210830-0419} (a).
We will prove that $\vec{N}(Du)\in L_p$, and \eqref{eqn-210901-0628} holds. For this, we fix any surface cube $\Delta_0\subset \p_l\cQ^T$ with $\diam(\Delta_0)<R_0$. In the following, we estimate $\|\vec{N}(Du)\|_{L_p(\Delta_0/4)}$ via Lemma \ref{lem-210827-0402}. Noting the embedding $L_s\hookrightarrow H^1$, for a surface cube $4\Delta\subset \Delta_0$,  let $w\in \cH^1_{2,loc}$ be the solution to
\begin{equation*}
	\begin{cases}
		w_t-\Delta w = 0  & \text{in }\, \cQ^T,\\
		\frac{\p w}{\p \vec{n}} = g_\cN 1_{4\Delta}  & \text{on }\, \cN^T,\\
		w = 0 & \text{on }\, \cD^T,\\
		w = 0 & \text{on }\, \{t=0\},\\
		\vec{N}(Dw) \in L_1(\p\cQ^T).
	\end{cases}
\end{equation*}
given by Theorem \ref{thm-210830-0419} (a), which satisfies
\begin{equation*}
	\frac{1}{|\Delta|}\|\vec{N}(Dw)\|_{L_1(\p_l\cQ^T)} \leq C(|g_\cN|^s)_{4\Delta}^{1/s}.
\end{equation*}
By the Hardy-Littlewood theorem applied to $(\vec{N}(Dw))^{1/2}$, we have
\begin{equation}\label{eqn-211017-1107}
	\fint_\Delta \big(\mathcal{M}(\vec{N}(Dw)^{1/2})\big)^2 \leq C(|g_\cN|^{s})_{4\Delta}^{1/s},
\end{equation}
where for $h\in L_{1,\text{loc}}(\p_l\cQ^T)$ and $X\in \partial_l\cQ$, we define
\begin{equation*}
	\cM(h)(X):= \sup_r\fint_{\Delta_r(X)} |h|1_{\p_l\cQ^T}\,d\sigma.
\end{equation*}
Now $v:=u-w$ satisfies
\begin{equation*}
	\begin{cases}
		v_t-\Delta v = 0  & \text{in }\, \cQ^T,\\
		\frac{\p v}{\p \vec{n}} = 0  & \text{on }\, \cN^T\cap 4\Delta,\\
		v = 0 & \text{on }\, \cD^T,\\
		\vec{N}(Dv) \in L_1(\p\cQ^T).
	\end{cases}
\end{equation*}
In Lemma \ref{lem-210827-0402}, we choose $f=g_\cN$ and
$$
F=\big(\mathcal{M}(\vec{N}(Du)^{1/2})\big)^2,\quad F_\Delta=\big(\mathcal{M}(\vec{N}(Dw)^{1/2})\big)^2,\quad R_\Delta=\big(\mathcal{M}(\vec{N}(Dv)^{1/2})\big)^2.
$$
Clearly, \eqref{eqn-210827-0358-1} holds. Also, \eqref{eqn-211017-1107} gives \eqref{eqn-210827-0358-3}. Hence, we are left to check \eqref{eqn-210827-0358-2}.

For this, we prove a reverse H\"older inequality for $\mathcal{M}\big(\vec{N}(Dv)^{1/2}\big)^2$.
\begin{proposition}\label{prop-210825-0426}
For any $q\in (1,(m+2)/(m+1))$, there exists $\theta=\theta(d,M,q)>0$ sufficiently small, such that if Assumption \ref{ass-210609-0500-3} ($\theta,m$) is satisfied, then
	\begin{equation*}
		\left(\fint_{\Delta} \big(\mathcal{M}(\vec{N}(Dv)^{1/2})\big)^{2q}\right)^{1/q} \leq C\fint_{8\Delta}\big(\mathcal{M}(\vec{N}(Dv)^{1/2})\big)^{2},
	\end{equation*}
where $C=C(d,M,q)$.
\end{proposition}
In the proof, we also need the truncated Hardy-Littlewood maximal function at the boundary: For $r=\diam(\Delta)$ and $h\in L_{1,\text{loc}}(\p_l\cQ^T)$, let
\begin{equation*}
	\mathcal{M}_{r,0}(h)(X):= \sup_{0<\tau<r} \fint_{\Delta_\tau(X)} |h|1_{\p_l\cQ^T},\quad \mathcal{M}_{r,\infty}(h)(X):= \sup_{\tau>r} \fint_{\Delta_\tau(X)} |h|1_{\p_l\cQ^T}.
\end{equation*}
\begin{proof}[Proof of Proposition \ref{prop-210825-0426}]
	 Recall the truncated non-tangential maximal function $\vec{N}_{r/2}(f)$ and $\vec{N}^{r/2}(f)$ defined in \eqref{eqn-210918-0600}. We estimate them separately.
	
	\textbf{Estimate of $\vec{N}_{r/2}(Dv)$.} By Lemma \ref{lem-210216-0333-1} (or, Lemma \ref{lem-210216-0333-2} for the Dirichlet case, Lemma \ref{lem-210706-0144} for the mixed case) and \eqref{eq6.54}, we have
	\begin{equation}\label{eqn-210827-0407}
		\big(\fint_{2\Delta}\vec{N}_{r/2}(Dv)^{q}\big)^{1/q} \leq C \big(\fint_{\cQ_{4r}} |Dv|^2\big)^{1/2} \leq C\fint_{\cQ_{8r}} |Dv| \leq C\fint_{8\Delta} \vec{N}(Dv).
	\end{equation}
Since the $r$-neighborhood of $\Delta$ is contained in $2\Delta$, by the Hardy-Littlewood maximal function theorem,
\begin{equation}\label{eqn-210827-0241-1}
	\begin{split}
		\left(\fint_{\Delta}\mathcal{M}_{r,0}\big((\vec{N}_{r/2}(Dv))^{1/2}\big)^{2q}\right)^{1/q}		&=\left(\fint_{\Delta}\mathcal{M}_{r,0}\big((1_{2\Delta}\vec{N}_{r/2}(Dv))^{1/2}\big)^{2q}\right)^{1/q} \\
		&\leq C\big(\fint_{2\Delta}\vec{N}_{r/2}(Dv)^{q}\big)^{1/q}\\	
		&\leq C\fint_{8\Delta} \vec{N}(Dv) \leq C\fint_{8\Delta} \mathcal{M}\big(\vec{N}(Dv)^{1/2}\big)^2.
	\end{split}
\end{equation}

\textbf{Estimate of $\vec{N}^{r/2}(Dv)$.}
For any $X\in 2\Delta$ and $Y\in \Gamma^{r/2}(X)$, set
\begin{equation*}
	\gamma(Y):=\{Z\in\p_l\cQ^T:Y\in \Gamma(Z)\}.
\end{equation*}
We aim to derive a lower bound for $|\gamma(Y)|$. Let $\widetilde{Y}\in\p_l\cQ^T$ be the point such that
$|Y-\widetilde{Y}| = d(Y)$.
By the triangle inequality,
\begin{equation*}
	d(Y) \geq |Y-Z|/(1+\alpha)\quad\forall Z\in \Delta_{\alpha d(Y)}(\widetilde{Y}),
\end{equation*}
which means
\begin{equation}\label{eqn-210826-0650}
\Delta_{\alpha d(Y)}(\widetilde{Y})\subset \gamma(Y).
\end{equation}
Again by the triangle inequality, we have $|X-\widetilde{Y}|\le |X-Y|+|Y-\widetilde{Y}|\le (2+\alpha)d(Y)$, and thus
\begin{equation}\label{eqn-210918-0549}
	\Delta_{\alpha d(Y)}(\widetilde{Y})\subset \Delta_{(2+2\alpha)d(Y)}(X).
\end{equation}
Using the inclusions \eqref{eqn-210826-0650} and \eqref{eqn-210918-0549}, we have for any $Z\in\p_l\cQ$,
\begin{equation}
                    \label{eq6.00}
	\begin{split}
		&|Dv(Y)|^{1/2} \leq \fint_{\Delta_{\alpha d(Y)}(\widetilde{Y})} (\vec{N}(Dv))^{1/2}\\
		&\leq
		\frac{|\Delta_{(2+2\alpha)d(Y)}(X)|}{|\Delta_{\alpha d(Y)}(\widetilde{Y})|}\fint_{\Delta_{(2+2\alpha)d(Y)}(X)} (\vec{N}(Dv))^{1/2}
		\\&\leq
		C_{\alpha,d,M} \fint_{\Delta_{(2+2\alpha)d(Y)}(X)} (\vec{N}(Dv))^{1/2}
		\\&\leq
		C_{\alpha,d,M} \frac{|\Delta_{(2+2\alpha)d(Y) +|X-Z|}(Z)|}{|\Delta_{(2+2\alpha)d(Y)}(X)|} \fint_{\Delta_{(2+2\alpha)d(Y)+|X-Z|}(Z)} (\vec{N}(Dv))^{1/2}.
	\end{split}	
\end{equation}
Furthermore, noting $Y\in\Gamma(X)$ and $|X-Y|\geq r/2$, we have
\begin{equation*}
	d(Y)=|Y-\widetilde{Y}| \geq |X-Y|/(1+\alpha) \geq r/(2+2\alpha).\end{equation*}
Taking $Z\in 8\Delta$ and the sup with respect to $Y\in \Gamma^{r/2}(X)$ in \eqref{eq6.00}, we get
\begin{equation*}
	\big(\vec{N}^{r/2}(Dv)(X)\big)^{1/2}\leq C\mathcal{M}_{r,\infty}((\vec{N}(Dv))^{1/2})(Z).
\end{equation*}
Taking sup on the left-hand side and average on the right-hand side, we have
\begin{equation}\label{eqn-210827-0241-2}
	\sup_\Delta\big(\mathcal{M}_{r,0}(\vec{N}^{r/2}(Dv))^{1/2}\big)^2
	\leq
	C\sup_{2\Delta}\vec{N}^{r/2}(Dv)
	\leq
	C\fint_{8\Delta}\big(\mathcal{M}_{r,\infty}((\vec{N}(Dv))^{1/2})\big)^2.
\end{equation}

\textbf{Estimate of $\mathcal{M}_{r,\infty}(\vec{N}(Dv)^{1/2})$.} This is the easy part. By the same reasoning with \cite[Lemma~A.4]{OB21}, for any function $f$, a constant $A>0$, and two points $X,Y$ with $|X-Y|<Ar$,
\begin{equation*}
	\mathcal{M}_{r,\infty}(f)(X)\approx_A \mathcal{M}_{r,\infty}(f)(Y).
\end{equation*}
Hence,
\begin{equation}\label{eqn-210827-0241-3}
	\sup_\Delta \mathcal{M}_{r,\infty}(\vec{N}(Dv)^{1/2}) \leq C \fint_{8\Delta}\mathcal{M}_{r,\infty}(\vec{N}(Dv)^{1/2}).
\end{equation}
Combining \eqref{eqn-210827-0241-1}, \eqref{eqn-210827-0241-2}, and \eqref{eqn-210827-0241-3}, the proposition is proved.
\end{proof}
It is not difficult to see that \eqref{eqn-210827-0358-2} follows from Proposition \ref{prop-210825-0426}, the definitions of $F$, $F_\Delta$, and $R_\Delta$, and \eqref{eqn-210827-0358-3}. Then by Lemma \ref{lem-210827-0402}, we have $\big(\mathcal{M}(\vec{N}(Du)^{1/2})\big)^2\in L_p(\Delta_0/4)$. Taking the extension at the beginning of Section \ref{sec-210713-0656} into consideration, by a covering argument, one can simply see that $\vec{N}(Du)\in L_p(\p_l\cQ^T)$ and the estimate \eqref{eqn-210901-0628} holds.

\begin{remark}
	In our previous paper \cite{DL20} regarding the elliptic case, there was a gap in the geometric argument when estimating $\vec{N}^{r/2}(Dv)$ on Page 22. A similar issue in \cite{TOB} and \cite{OB13} was pointed out and resolved in \cite{OB21}. More precisely, in \cite{DL20} we used for any surface cube $Q$ and $x\in Q$, a pointwise bound $\vec{N}^{r_Q/2}(Dv)(x) \leq (\vec{N}(Dv))_{2Q}$, which is not correct. Such problem can be fixed by using the Hardy-Littlewood maximal function on the boundary, as described in the proof above.
\end{remark}

\subsection{\texorpdfstring{$L_{1+\varepsilon}$}{L} estimate without flatness assumption}
The proof is almost the same with that of Theorem \ref{thm-211009-0913} in Section \ref{sec-210901-0632}. The only difference is that since we do not have the flatness of $\Lambda$, here we can only take $q\in(1,p_0/2)$ in \eqref{eqn-210827-0407}. Eventually we reach $\vec{N}(Du)\in L_{1+\epsi}$ and the corresponding estimate, which is Theorem \ref{thm-210830-0419} (c).

\appendix
\section{A geometric lemma}\label{app-210919-1049}
\begin{lemma}\label{lem-210629-1247}
	For any $\epsi_0>0$, we can find $\theta=\theta(d,m,M,\epsi_0)$ small enough, such that if Assumption \ref{ass-210609-0500-3} (b) with parameters ($\theta,m$) is satisfied, then Assumption \ref{ass-210730-0540} ($\epsi_0$) holds.
\end{lemma}
\begin{proof}
	In the following, we only prove \eqref{eqn-210622-0316-1}. The proof for \eqref{eqn-210622-0316-2} is similar. Without loss of generality, we take the center $X=0$.
	
	The lower bound is straightforward: as long as $\theta\leq 1/2$ in Assumption \ref{ass-210609-0500-3} ($\theta,m$),
	\begin{equation*}
		\int_{\Delta_r}\delta^s\,d\sigma \geq \int_{\Delta_r, |y^2-\phi|\geq r/2}\delta^s\,d\sigma \approx \int_{\Delta_r, |y^2-\phi|\geq r/2}r^s\,d\sigma \approx r^{s+d+1}.
	\end{equation*}
	This proves the lower bound in \eqref{eqn-210622-0316-1}.
	
	Now we prove the upper bound. When $s\geq 0$, by Assumption \ref{ass-210609-0500-3},
$$
\delta(Y) \leq
C|y^2-\phi(y^3,\ldots,y^{m+2})-\theta r|+C|y^2-\phi(y^3,\ldots,y^{m+2})+\theta r|,
$$
where $C=C(M)$.
Now \eqref{eqn-210622-0316-1} follows as
	\begin{equation*}
		\int_{\Delta_r}\delta^s\,d\sigma \leq C\int_{\Delta_r}(
|y^2-\phi(y^3,\ldots,y^{m+2})-\theta r|^s+|y^2-\phi(y^3,\ldots,y^{m+2})+\theta r|^s
) \,d\sigma \approx r^{d+1+s}.
	\end{equation*}
	Again, in this case no smallness of $\theta$ is needed.
	
	We are left with the case $s<0$.
	First, take a decomposition $\Delta_r=\bigcup_{k=1}^\infty \Delta_r^{(k)}$, where
		\begin{equation}\label{eqn-210929-0801}
			\Delta_r^{(k)}:=\{Y\in\Delta_r,\, \delta(Y) \in [(2\theta)^{k} r, (2\theta)^{k-1} r)\}.
	\end{equation}
	Since $\Lambda$ is $(\theta,m)$-flat, for any $r<R_0$, we see that $\Delta_r(X)\cap\Lambda$ can be covered by $C_0/\theta^d$ surface cubes with centers at $\Lambda$ and radii $2\theta r$, where $C_0$ depends only on $d$ and $M$. As long as $\theta< 1/2$, we can iterate this to find for each $k\geq 1$ a covering by $t_k$ surface cubes  with centers at $\Lambda$, radii $(2\theta)^k r$, and $t_k\leq (C_0/\theta^d)^k$, i.e.,
	\begin{equation*}
		\Delta_r\cap\Lambda \subset \bigcup_{i=1}^{t_k} \Delta_{(2\theta)^k r}(X_i^{(k)}),\quad X_i^{(k)}\in\Lambda.
	\end{equation*}
	By the triangle inequality,
	\begin{equation*}
		\Delta_r^{(k+1)}\subset \{Y\in\Delta_r,\, \delta(Y) < (2\theta)^k r\}\subset \bigcup_{i=1}^{t_k} \Delta_{2(2\theta)^k r}(X_i^{(k)}).
	\end{equation*}
Hence, noting $s<0$ and \eqref{eqn-210929-0801},
	\begin{align}
		\int_{\Delta_r}\delta^s &\leq \sum_{k=1}^{\infty} \int_{\Delta_r^{(k)}}\delta^s
\leq
		C(2\theta)^sr^{d+1+s} + \sum_{k=1}^{\infty}\sum_{i=1}^{t_k}
\int_{\Delta_{2(2\theta)^k r}(X_i^{(k)})\cap \Delta_r^{(k+1)}}\delta^s \nonumber
		\\&\leq
			C(2\theta)^sr^{d+1+s} + \sum_{k=1}^{\infty}\sum_{i=1}^{t_k}	
		|\Delta_{2(2\theta)^k r}(X_i^{(k)})|((2\theta)^{k+1} r)^s \notag
			\\&\leq
			C(2\theta)^sr^{d+1+s} + \sum_{k=1}^{\infty} (C_0/\theta^d)^k C\big(2(2\theta)^k r\big)^{d+1} \big((2\theta)^{k+1}r\big)^s	\label{eqn-210929-0806}	
			\\&\leq
			C(2\theta)^sr^{d+1+s} + C(2r)^{d+1+s}\theta^s \sum_{k=1}^\infty (2^{d+1+s}C_0\theta^{1+s})^k. \nonumber		
	\end{align}
Here in \eqref{eqn-210929-0806}, we used $t_k\leq (C_0/\theta^d)^k$.
	We are left to choose $\theta$ small enough such that
	\begin{equation*}
		2^{d+1+s}C_0\theta^{1+s} \leq 2^{d+1}C_0\theta^{\epsi_0}<1.
	\end{equation*}
The lemma is proved.
\end{proof}

\section{Proof of Lemma \ref{lem-211105-0444}}\label{app-210824-1212}
Now we give the construction of the cut-off function in Section \ref{sec-210817-0523}. First, for the closed set $\overline{\cD^S}\subset \mathbb{R}^{d+1}$, there exists a (parabolic) regularized distance $\rho_{\cD^S}\in C^{0,1}(\bR^{d+1})\cap C^\infty(\bR^{d+1}\setminus\overline{\cD^S})$, with
\begin{align}
	C^{-1}\leq \frac{\rho_{\cD^S}(X)}{\dist(X,\cD^S)} \leq C\quad\forall X\in \mathbb{R}^{d+1}\setminus\overline{\cD^S},\label{eqn-210823-1148-1}
	\\
	|\p^{\vec{\alpha}}_x \rho_{\cD^S}(X)| \leq C_{\vec{\alpha}}\rho_{\cD^S}(X)^{1-|\alpha|}\quad \forall X\in \mathbb{R}^{d+1}\setminus\overline{\cD^S},\,\,\vec{\alpha}=(\alpha_1,\ldots\alpha_d)\in \mathbb{N}^d,\label{eqn-210823-1148-2}
	\\
	|\p^k_t \rho_{\cD^S}(X)| \leq C_k\rho_{\cD^S}(X)^{1-2k}\quad \forall X\in \mathbb{R}^{d+1}\setminus\overline{\cD^S},\,\,k\in\mathbb{N}.\label{eqn-210823-1148-3}
\end{align}
Such $\rho_{\cD^S}$ can be constructed by decomposing $\bR^{d+1}\setminus\overline{\cD^S}$ into Whitney cubes as in \cite[Theorem~VI.2.2]{MR0290095}. Similarly, we can construct $\rho_{\cN^S}$, the regularized distance to $\overline{\cN^S}$. Now, for a usual cut-off function $\varphi$ on $\bR$ with
\begin{equation*}
	\varphi|_{(-\infty,1)} = 1,\quad\varphi|_{(2,\infty)}=0,\quad\text{and}\,\,|D\varphi|\leq C,
\end{equation*}
we define
\begin{equation*}
	\eta = 1 - \varphi\big(\frac{\rho_{\cD^S}}{\rho_{\cN^S}}\big)\varphi\big(\frac{\rho_{\cD^S}}{h}\big).
\end{equation*}
By using \eqref{eqn-210823-1148-1}, \eqref{eqn-210823-1148-2}, and \eqref{eqn-210823-1148-3}, it is not difficult to check that such $\eta$ satisfies the desired properties.



\def\cprime{$'$}

\end{document}